%% file: A-CODER.tex
\documentclass{article}

\usepackage[english]{babel}
\usepackage[letterpaper,top=2cm,bottom=2cm,left=3cm,right=3cm,marginparwidth=1.75cm]{geometry}

\usepackage{booktabs}

\input{math_commands_arxiv}

\title{Accelerated Cyclic Coordinate Dual Averaging with Extrapolation for Composite Convex Optimization}

\author{
    Cheuk Yin Lin \thanks{Department of Computer Sciences, University of Wisconsin-Madison. E-mail:
    \href{mailto:cylin@cs.wisc.edu}{\texttt{cylin@cs.wisc.edu}}, 
    \href{mailto:chaobing.song@wisc.edu}{\texttt{chaobing.song@wisc.edu}}, 
    \href{mailto:jelena@cs.wisc.edu}{\texttt{jelena@cs.wisc.edu}}}
    \and Chaobing Song \footnotemark[1]
    \and Jelena Diakonikolas \footnotemark[1]
}

\begin{document}

\maketitle

\begin{abstract}
Exploiting partial first-order information in a cyclic way is arguably the most natural strategy to obtain scalable first-order methods. However, despite their wide use in practice, cyclic schemes are far less understood from a theoretical perspective than their randomized counterparts. Motivated by a recent success in analyzing an extrapolated cyclic scheme for generalized variational inequalities, we propose an \emph{Accelerated Cyclic Coordinate Dual Averaging with Extrapolation} (A-CODER) method for composite convex optimization, where the objective function can be expressed as the sum of a smooth convex function accessible via a gradient oracle and a convex, possibly nonsmooth, function accessible via a proximal oracle. We show that A-CODER attains the optimal convergence rate with improved dependence on the number of blocks compared to prior work. Furthermore, for the setting where the smooth component of the objective function is expressible in a finite sum form, we introduce a variance-reduced variant of A-CODER, VR-A-CODER, with state-of-the-art complexity guarantees. Finally, we demonstrate the effectiveness of our algorithms through numerical experiments.
\end{abstract}

\section{Introduction}\label{sec:intro}
Block coordinate descent methods are broadly used in machine learning due to their effectiveness on large datasets brought by cheap iterations requiring only partial access to problem information~\cite{wright2015coordinate,nesterov2012efficiency}. They are frequently applied to problems such as feature selection~\cite{wu2008coordinate,friedman2010regularization, mazumder2011sparsenet}, empirical risk minimization~\cite{nesterov2012efficiency,zhang2015stochastic, lin2015accelerated,allen2016even,alacaoglu2017smooth,gurbuzbalaban2017cyclic,diakonikolas2018alternating}, and in distributed computing \cite{liu2014asynchronous, fercoq2015accelerated,richtarik2016parallel}. In the more recent literature, coordinate updates on either the primal or the dual side in primal-dual settings have been used to attain variance-reduced guarantees in finite sum settings~\cite{chambolle2018stochastic,alacaoglu2017smooth,alacaoglu2020random,song2020variance,song2021coordinate}.  

Most of the existing theoretical results for (block) coordinate-type methods have been established for algorithms that select coordinate blocks to be updated by random sampling  without replacement~\cite{nesterov2012efficiency,wright2015coordinate,chambolle2018stochastic,alacaoglu2017smooth,alacaoglu2020random,song2020variance,song2021coordinate, zhang2015stochastic, lin2015accelerated,allen2016even,diakonikolas2018alternating}. Such methods are commonly referred to as the randomized block coordinate methods (RBCMs). What makes these methods particularly appealing from the aspect of convergence analysis is that the gradient evaluated on the sampled coordinate block can be related to the full gradient, by taking the expectation over the random choice of a coordinate block. 

An alternative class of block coordinate methods is the class of cyclic block coordinate methods (CBCMs), which update blocks of coordinates in a cyclic order. CBCMs are frequently used in practice due to often superior empirical performance compared to RBCMs~\cite{beck2013convergence,chow2017cyclic,sun2019worst} and are also  part of standard software packages for high-dimensional computational statistics such as  GLMNet \cite{friedman2010regularization} and SparseNet \cite{mazumder2011sparsenet}. However, CBCMs have traditionally been considered much more challenging to analyze than RBCMs. 

The first convergence rate analysis of CBCMs for smooth convex optimization problems, obtained by \citet{beck2013convergence}, relied on relating the partial coordinate blocks of the gradient to the full gradient. For this reason, the dependence of iteration complexity on the number of coordinate blocks in  \citet{beck2013convergence} scaled linearly and as a square root for vanilla CBCM and its accelerated variant, respectively. Such a high dependence on the number of blocks (equal to the dimension in the coordinate case) makes the complexity guarantee of CBCMs seem worse than not only RBCMs but even full gradient methods such as gradient descent and the fast gradient method of~\citet{nesterov1983method}, bringing into question their usefulness. This is further exacerbated by a result that shows that such a high gap in complexity does happen in the worst case~\cite{sun2019worst}, prompting research that would explain the gap between the theory and practice of CBCMs. However, most of the results that improved the dependence on the number of blocks only did so for structured classes of convex quadratic problems~\cite{wright2020analyzing,lee2019random,gurbuzbalaban2017cyclic}.   

On the other hand, a very recent work in~\citet{song2021cyclic} introduced an extrapolated CBCM for variational inequalities whose complexity guarantee does not involve explicit dependence on the number of blocks. This result is enabled by a novel Lipschitz condition introduced in the same work. While the result from~\citet{song2021cyclic} applies to convex minimization settings as a special case, the obtained convergence rates are not accelerated. Our main motivation in this work is to close this convergence gap by providing accelerated extrapolated CBCMs for convex composite minimization. 


\subsection{Contributions}

We study the following composite convex problem
\begin{align}\label{eq:main-problem}\tag{P}
\min_{\vx\in\sR^d}\Big\{\bar{f}(\vx) = f(\vx) + g(\vx)\Big\},
\end{align}
where $f$ is smooth and convex and $g$ is proper, (possibly strongly) convex, and lower semicontinuous. This is a standard and broadly studied setting of structured nonsmooth optimization; see, e.g,~\citet{beck2009fast,nesterov2007gradient} and the follow-up work. To further make the problem amenable to optimization via block coordinate methods, we assume that $g$ is block separable, with each component function admitting an efficiently computable prox operator (see Section~\ref{sec:prelims} for a precise statement of the assumptions).

Similar to \citet{song2021cyclic}, we define a summary Lipschitz constant $L$ of $f$ obtained from Lipschitz conditions of individual blocks. Our summary Lipschitz condition is similar to that of \citet{song2021cyclic} (although not exactly the same) and enjoys the same favorable properties as the condition introduced in that paper; see Section~\ref{sec:prelims} for more details.

We introduce a new accelerated cyclic algorithm for \eqref{eq:main-problem} whose full gradient oracle complexity (number of full gradient passes or, equivalently, number of full cycles) is of the order $O\Big(\min\Big\{\sqrt{\frac{L}{\epsilon}}\|\vx_0 - \vx^*\|_2, \; \sqrt{\frac{L}{\gamma}}\log(\frac{L\|\vx_0 - \vx^*\|_2}{\epsilon}\Big\}\Big),$ where $\gamma$ is the strong convexity parameter of $g$ (equal to zero if $g$ is only convex, by convention), $\vx^*$ is an optimal solution to \eqref{eq:main-problem}, and $\vx_0 \in \dom(g)$ is an arbitrary initial point. This complexity result matches the gradient oracle complexity of the fast gradient method~\cite{nesterov1983method}, but with the traditional Lipschitz constant being replaced by the Lipschitz constant introduced in our work. In the very worst case, this constant is no higher than $\sqrt{m}$ times the traditional one, where $m$ is the number of blocks, giving an $m^{1/4}$ improvement in the resulting complexity over the accelerated cyclic method from \citet{beck2013convergence}. Even in this worst case, the obtained improvement in the dependence on the number of blocks is the first such improvement for accelerated methods since the work of \citet{beck2013convergence}. We note, however, that for both synthetic data and real data sets and on an example
problem where both Lipschitz constants are explicitly computable, our Lipschitz constant is within a small constant factor (smaller than 1.5) of the traditional one (see Figure~\ref{fig:L-compare}, Table~\ref{table:L-compare}, and the related discussion in Section~\ref{sec:prelims}).  

Some key ingredients in our analysis are the following. First, we construct an estimate of the optimality gap we want to bound, where we replace the gradient terms with a vector composed of partial, or block, extrapolated gradient terms evaluated at intermediate points within a cycle. Crucially, we show that the error introduced by doing so can be controlled and bounded via our Lipschitz condition. An auxiliary result allowing us to carry out the analysis and appropriately bound the error terms resulting from our approach is Lemma~\ref{lemma:Lip-quad-ub}, which shows that our Lipschitz condition translates into inequalities of the form     
\begin{align*}
    f(\vy) - f(\vx) \leq\;& \innp{\nabla f(\vx), \vy - \vx} + \frac{L}{2} \|\vy - \vx\|^2, \\ 
    \|\nabla f(\vy) - \nabla f(\vx)\|^2 \le\;& 2 L (f(\vy) - f(\vx) - \langle\nabla f(\vx), \vy - \vx\rangle),
\end{align*}
similar to the standard inequalities that hold for the traditional, full-gradient, Lipschitz constant. Finally, we note that the accelerated algorithm that we introduce is novel even in the single block (i.e., full-gradient) setting, due to the employed gradient extrapolation.

We further consider the finite sum setting, where $f$ is expressible as $f(\vx) = \frac{1}{n} \sum_{t=1}^n f_t(\vx),$ and where $n$ is typically very large. We then propose a variance-reduced variant of our accelerated method, which further reduces the full gradient oracle complexity to $O\Big(\min\Big\{\sqrt{\frac{L}{n \epsilon}}\|\vx_0 - \vx^*\|_2, \; \sqrt{\frac{L}{n \gamma}}\log(\frac{L\|\vx_0 - \vx^*\|_2}{\epsilon}\Big\}\Big)$.  
The variance reduction that we employ is of the SVRG type~\cite{johnson2013accelerating}. While following a similar approach as the basic accelerated algorithm described above, the analysis in this case turns out to be much more technical, due to the need to simultaneously handle error terms arising from variance reduction as well as the error terms arising from the cyclic updates.
{ Through utilizing the novel smoothness properties obtained in Lemma~\ref{lemma:Lip-quad-ub} specific to convex minimization, we are able to obtain the desired error bounds without using the additional point extrapolation step in the gradient estimator as~\citet{song2021cyclic}, but rather only with an SVRG estimator. This important change paves a path to achieving accelerated convergence rates while also simplifying the implementation of our algorithms.}

{ Last but not least, we demonstrate the practical efficacy of our novel accelerated algorithms A-CODER and VR-A-CODER through numerical experiments, comparing against other relevant block coordinate descent methods. The use of A-CODER and VR-A-CODER achieves faster convergence in primal gap with respect to both the number of full-gradient evaluations and wall-clock time.}

\subsection{Further Discussion of Related Work}

As discussed at the beginning of this section, cyclic block coordinate methods constitute a fundamental class of optimization methods whose convergence is not yet well understood. In the worst case, the full gradient oracle complexity of vanilla cyclic block coordinate gradient update is worse than that of vanilla gradient descent, by a factor scaling with the number of blocks $m$ (equal to the dimension in the coordinate case)~\cite{sun2019worst,beck2013convergence}. Since the initial results providing such an upper bound~\cite{sun2019worst}, there were no improvements on the dependence on the dependence on the number of blocks in the convergence guarantees of cyclic methods until the very recent work of \citet{song2021cyclic}, which in the worst case improves the dependence on $m$ by a factor $\sqrt{m}.$ Our work further contributes to this line of work by improving the dependence on $m$ in accelerated methods from $\sqrt{m}$ to $m^{1/4}$ in the worst case.  

In the finite-sum settings, variance reduction has been widely explored; e.g., in~\citet{johnson2013accelerating,defazio2014saga,allen2017katyusha,reddi2016stochastic,lei2017non,song2020variance,schmidt2017minimizing} for the case of full-gradient methods and in~\citet{Chen2016AcceleratedSB, Lei2018AsynchronousVB} for randomized block coordinate methods. However, variance reduced schemes for cyclic methods are much more rare, with nonasymptotic guarantees being obtained very recently for the case of variational inequalities~\cite{song2021cyclic} and nonconvex optimization~\cite{Cai2022CyclicBC, Xu2014BlockSG}. We are not aware of any existing variance reduced results for accelerated cyclic block coordinate methods. 

\subsection{Outline of the Paper}

Section~\ref{sec:prelims} introduces the necessary notation and background and outlines our main problem assumptions. Section~\ref{sec:acoder} introduces the A-CODER algorithm and outlines the analysis. For space constraints, the full convergence analysis of A-CODER is provided in Appendix~\ref{appx:omitted-proofs-acoder}. Section~\ref{sec:vr-acoder} presents VR-A-CODER and outlines its convergence analysis, while the full technical details are deferred to Appendix~\ref{appx:omitted-proofs-acoder-vr}. Finally, Section~\ref{sec:num-exp}  provides numerical experiments for our results and concludes the paper with a discussion. 

\section{Notation and Preliminaries}\label{sec:prelims}
For a positive integer $K,$ we use $[K]$ to denote the set $\{1,2,\ldots, K\}.$ We consider the $d$-dimensional Euclidean space $(\sR^d, \|\cdot\|),$ where $\|\cdot\| = \sqrt{\innp{\cdot, \cdot}}$ denotes the Euclidean norm, $\innp{\cdot, \cdot}$ denotes the (standard) inner product, and $d$ is assumed to be finite. 
Throughout the paper, we assume that there is a given partition of the set $\{1, 2, \dots, d\}$ into sets $\gS^j$, $j \in \{1, \dots, m\},$ where $|\gS^j| = d^j > 0.$ For convenience of notation, we assume that sets $\gS^j$ are comprised of consecutive elements from $\{1, 2, \dots, d\}$, that is,  $\gS^1 = \{1, 2, \dots, d^1\},$ $\gS^2 = \{d^1 + 1, d^1 + 2, \dots, d^1 + d^2\},\dots, \gS^m = \{\sum_{j=1}^{m-1}d^j + 1, \sum_{j=1}^{m-1}d^j + 2, \dots, \sum_{j=1}^{m}d^j\}$. This assumption is without loss of generality, as all our results are invariant to permutations of the coordinates (though the value of the Lipschitz constant of the gradients defined in our work depends on the ordering of the coordinates; see Assumption~\ref{assmpt:Lip-const}). For a vector $\vx \in \rr^d$, we use $\vx^{(j)}$ to denote its coordinate components indexed by $\gS^j$. Similarly for a gradient $\nabla f$ of a function $f:\sR^d\rightarrow\sR$, we use $\nabla^{(j)} f$ to denote its coordinate components indexed by $\gS^j.$
We use $(\;\cdot\;)_{\ge j}$ to denote an operator for vectors and square matrices that \emph{replaces the first $j-1$ elements of rows and columns with zeros}, i.e., keeping elements with indices $\ge j$ the same, otherwise zeros.



Given a proper, convex, lower semicontinuous function $g: \sR^d \to \sR \cup \{+\infty\},$ we use $\partial g(\vx)$ to denote the subdifferential set (the set of all subgradients) of $g$. Of particular interests to us are functions $g$ whose proximal operator (or resolvent), defined by
\begin{equation}\label{eq:prox-op}
    \mathrm{prox}_{\tau g}(\vu) := \argmin_{\vx \in \sR^d}\Big\{\tau g(\vx) + \frac{1}{2 }\|\vx - \vu\|^2\Big\}
\end{equation}
is efficiently computable for all $\tau > 0$ and $\vu \in \sR^d.$ 
To unify the cases in which $g$ are convex and strongly convex respectively, we say that $g$ is $\gamma$-strongly convex with modulus $\gamma \geq 0,$ if for all $\vx, \vy \in \sR^d$ and $g'(\vx)\in \partial g(\vx)$, 
\begin{equation*}
    g(\vy) \geq g(\vx) +  \innp{g'(\vx), \vy - \vx} + \frac{\gamma}{2}\|\vy - \vx\|^2. 
\end{equation*}

\paragraph{Problem definition.} 
We consider Problem~\eqref{eq:main-problem}, 
under the following assumptions.



\begin{assumption}\label{assmpt:strongly-convex}
$g(\vx)$ is $\gamma$-strongly convex, where $\gamma\ge 0$, and block-separable over coordinate sets $\{\gS^j\}_{j=1}^m:$ $g(\vx) = \sum_{j=1}^m g^j(\vx^{(j)})$.
Each $g^j(\vx^{(j)})$ for $j \in [m]$ admits an efficiently computable proximal operator. 
\end{assumption} 

\begin{assumption}\label{assmpt:Lip-const} 
There exist positive semidefinite matrices $\{\mQ^1,\mQ^2,\ldots, \mQ^m\}$ such that $\nabla^{(j)} f(\cdot)$  is $1$-Lipschitz continuous w.r.t.~the seminorm $\|\cdot\|_{\mQ^j},$ i.e., $\forall \vx, \vy \in \sR^d,$
\begin{equation}\label{eq:block-Lipschitz}
\|\nabla^{(j)} f(\vx) - \nabla^{(j)} f(\vy)\|^2 \le \|\vx - \vy\|_{\mQ^j}^2,
\end{equation}
where $\|\vx - \vy\|_{\mQ^j}^2 := (\vx-\vy)^T\mQ^j(\vx-\vy)$ is the Mahalanobis (semi)norm. Moreover, we define a new Lipschitz constant $L$ such that $L^2 = 2 \|\tilde{\mQ}\| \; < \infty$ where $\tilde{\mQ} = \sum_{j=1}^m \left[ (\mQ^j)_{\ge j} + (\mQ^j)_{\ge j+1} \right]$.
\end{assumption}

Observe that when $f$ is $M$-smooth in a traditional sense (i.e., when $f$ has $M$-Lipschitz gradients w.r.t.~the Euclidean norm), Assumption~\ref{assmpt:Lip-const} can be trivially satisfied using $\mQ^j = M \mI$ for all $j \in [m],$ where $\mI$ is the identity matrix. Consequently, it can be argued that $L \leq 2\sqrt{m}M$~\citep{song2021cyclic}; however, we show that this bound is much tighter in practice as illustrated in Figure~\ref{fig:L-compare} and in Table~\ref{table:L-compare}. In particular, we follow the experiments in~\citet{song2021cyclic} and show empirically that the standard Lipschitz constant $M$ and our new Lipschitz constant $L$ scale within the same factor for both synthetic and real data.

\begin{figure}[t!]
    \centering
    {\includegraphics[width=0.48\textwidth]{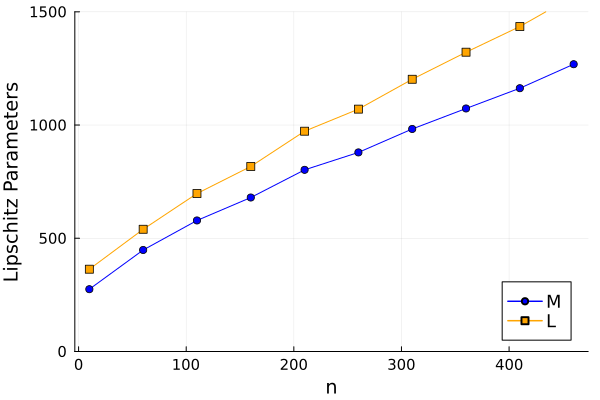}\label{fig:L-compare-time-n}} 
    {\includegraphics[width=0.48\textwidth]{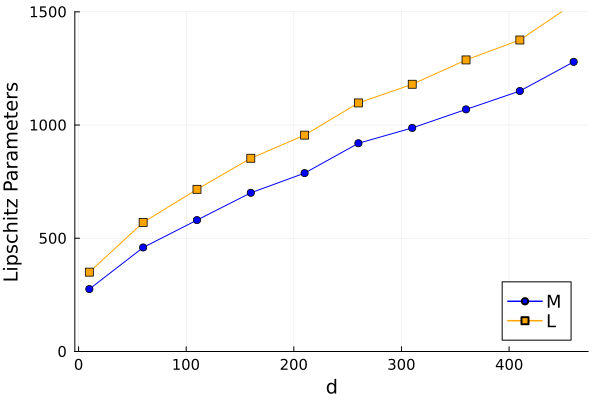}\label{fig:L-compare-time-d}}
    \caption{Comparisons of Lipschitz constants for elastic-net problems on synthetic datasets, where $M$ denotes the commonly known Lipschitz constant and $L$ is our new Lipschitz constant as defined in Assumption~\ref{assmpt:Lip-const}.}
    \label{fig:L-compare}
\end{figure}

\begin{table}[t!]
\caption{Comparisons of Lipschitz constants for elastic-net problems on LibSVM datasets. $M$ is the classical gradient  Lipschitz  constant and $L$ is our novel smoothness constant. We use each coordinate as a block, i.e., $m = d$.}
\label{table:L-compare}
\vskip 0.15in
\begin{center}
\begin{small}
\begin{sc}
\begin{tabular}{lccc}
\toprule
Dataset & \#Features & $M$ & $L$ \\
\midrule
sonar & 60 & 12.5 & 15.8 \\
colon & 2000 & 310.6 & 394.7 \\
a9a & 123 & 6.1 & 7.7 \\
phishing & 68 & 0.60 & 0.76 \\ 
madelon & 500 & 1.2 & 1.5 \\
\bottomrule
\end{tabular}
\end{sc}
\end{small}
\end{center}
\vskip -0.1in
\end{table}

\section{Accelerated Cyclic Algorithm}\label{sec:acoder}

In this section, we introduce and analyze A-CODER, whose pseudocode is provided in Algorithm~\ref{alg:acc-coder}. A-CODER can be seen as a Nesterov-style accelerated variant of CODER, previously introduced for solving variational inequalities by~\citet{song2021cyclic}. A-CODER is related to other accelerated algorithms in the following sense. In the case of a single block ($m = 1$) and when gradient extrapolation is not used (i.e., when $\vq_k = \vp_k$), A-CODER reduces to a generalized variant of AGD+ \citep{cohen2018acceleration,diakonikolas2021complementary} or the method of similar triangles~\citep{gasnikov2018universal}. 
The analysis of A-CODER follows the general gap bounding argument~\citep{diakonikolas2019approximate,song2021unified} and it is based on three key ingredients: (i) gradient extrapolation, which enables the use of partial information about the gradients within a full epoch of cyclic updates, (ii) Lipschitz condition for the gradients based on the Mahalanobis norm as defined in Assumption~\ref{assmpt:Lip-const}, and (iii) upper and lower bounds on the difference between the function and its linear approximation that are compatible with the gradient Lipschitz condition that we use, as stated in Lemma~\ref{lemma:Lip-quad-ub}.

\begin{restatable}{lemma}{lemmaLipquadub} \label{lemma:Lip-quad-ub}
Let $f: \sR^d \rightarrow \sR$ be a convex and smooth function whose gradients satisfy Assumption~\ref{assmpt:Lip-const}. Then,  $\forall \vx, \vy \in \sR^d:$
\begin{align*}
    f(\vy) - f(\vx) \leq\;& \innp{\nabla f(\vx), \vy - \vx} + \frac{L}{2} \|\vy - \vx\|^2, \\ 
    \|\nabla f(\vy) - \nabla f(\vx)\|^2 \le\;& 2 L (f(\vy) - f(\vx) - \langle\nabla f(\vx), \vy - \vx\rangle).
\end{align*}
\end{restatable}

We now derive the A-CODER algorithm. We define $\{a_k\}_{k\ge 1}$ and $\{A_k\}_{k\ge 1}$ to be sequences of positive numbers with $A_k =\sum_{i=1}^k a_i , a_0 = A_0 = 0.$ Let $\{\vx_k\}_{k\ge 0}$  be an arbitrary sequence of points in $\dom(g)$. Our goal here is to bound the function value gap $\bar{f}(\vy_k) - \bar{f}(\vu)$ above for all $\vu\in\dom(g)$. Towards this goal, we define an estimation sequence $\psi_k$ recursively by $\psi_0(\vu) = \frac{1}{2}\|\vu - \vx_0\|^2$ and
\begin{equation}\notag
    \psi_k(\vu) := \psi_{k-1}(\vu) + a_k \big( f(\vx_k)  + \innp{\vq_k, \vu - \vx_k} + g(\vu)\big)
\end{equation}
for $k \geq 1$.
Meanwhile, $\vv_k$ and $\vy_k$ are defined as $\vv_k := \argmin_{\vu\in \sR^d}\psi_k(\vu)$ and $\vy_k := \frac{1}{A_k}\sum_{i=1}^k a_i \vv_i$ respectively.
We start our analysis by characterizing the gap function in the following lemma.
\begin{restatable}{lemma}{lemmaacodergapfn}\label{lemma:a-coder-gap-fn}
For any $\vu \in \sR^d$ and any sequence of vectors $\{\vq_i\}_{i\geq 1},$ we have
\begin{equation}
A_k(\bar{f}(\vy_k) - \bar{f}(\vu)) \le \sum_{i=1}^k E_i(\vu) + \frac{1}{2}\|\vu - \vx_0\|^2-\frac{1 + A_k \gamma}{2}\|\vu - \vv_k\|^2,
\end{equation}
where 
\begin{align}
E_i(\vu) =\;&  A_i  (f(\vy_i) - f(\vx_i))  - A_{i-1}  (f(\vy_{i-1}) - f(\vx_i)) \nonumber \\ 
& - a_i \innp{\vq_i, \vv_i - \vx_i}   + a_i \innp{\nabla f(\vx_i) -\vq_i,  \vx_i - \vu} \nonumber \\
& - \frac{1+A_{i-1}\gamma}{2}\|\vv_i - \vv_{i-1}\|^2. \label{eq:acc-E_i}   
\end{align}
\end{restatable}
%
%

\begin{algorithm}[t!]
\caption{Accelerated Cyclic cOordinate Dual avEraging with extRapolation (A-CODER)} \label{alg:acc-coder}
\begin{algorithmic}[1]
\STATE \textbf{Input:} $\vx_0\in\mathrm{dom}(g)$, $\gamma \geq 0$, $L>0$, $m$, $\{S^1, \dots, S^m\}$
\STATE \textbf{Initialization:} $\vx_{-1} = \vx_{0} = \vv_{-1} = \vv_0 = \vy_0$; $\vp_0 = \nabla f(\vx_0)$; $\vz_0 = \vzero$; $a_0= A_0 = 0$
\FOR {$k = 1$ to $K$} 
    \STATE Set $a_k > 0$ be largest value s.t. $\frac{a_k^2}{A_{k}} \le \frac{2 \pr{1 + A_{k-1} \gamma}}{5 L}$ where $A_k = A_{k-1} + a_k$
    \STATE $\vx_k = \frac{A_{k-1}}{A_k}\vy_{k-1} + \frac{a_k}{A_k}\vv_{k-1}$
    \FOR {$j = m$ to $1$}
        \STATE $\vp_k^{(j)} = \nabla^{(j)} f(\vx^{(1)}_{k}, \ldots, \vx^{(j)}_{k}, \vy^{(j+1)}_{k},\ldots,   \vy^{(m)}_{k})$
        \STATE $\vq^{(j)}_k = \vp_k^{(j)} + \frac{a_{k-1}}{a_k}(\nabla^{(j)}f(\vx_{k-1}) - \vp_{k-1}^{(j)})$
        \STATE $\vz_k^{(j)} = \vz_{k-1}^{(j)} + a_k\vq^{(j)}_k$
        \STATE $\vv_{k}^{(j)} = \mathrm{prox}_{A_k g^j}(\vx_0^{(j)} - \vz_k^{(j)})$
        \STATE $\vy_k^{(j)} = \frac{A_{k-1}}{A_k}\vy_{k-1}^{(j)} + \frac{a_k}{A_k}\vv_k^{(j)}$
    \ENDFOR
\ENDFOR
\STATE \textbf{return} $\vv_K, \vy_K$
\end{algorithmic}	
\end{algorithm}
%
%
Lemma~\ref{lemma:a-coder-gap-fn} applies to an arbitrary algorithm that satisfies its assumptions. From now on, we make the analysis specific to A-CODER (Algorithm~\ref{alg:acc-coder}).
In Lemma~\ref{lemma:a-coder-gap-fn}, $\{E_i(\vu)\}$ are the error terms that we need to bound above. If $\sum_{i=1}^k E_i(\vu) \le \frac{1 + A_k\gamma}{2}\|\vu - \vv_k\|^2,$ then we get the desired $1/A_k$ rate. To this end, we bound each term $E_k(\vu)$ in Lemma \ref{lemma:acoder-gap-change} by using the extrapolation direction $\vq_k$, the definition of $\vy_k, \vx_k$ and the parameter setting of $a_k$.
\begin{restatable}{lemma}{lemmaacodergapchange} \label{lemma:acoder-gap-change}
Let $\vx_0 \in \dom(g)$ be an arbitrary initial point and consider the updates in Algorithm~\ref{alg:acc-coder}. If, for $k \geq 1$, $\frac{{a_k}^2}{A_k} \leq \frac{2 (1 + A_{k-1}\gamma)}{5 L}$, then $\forall \vu$,
\begin{align*}
    E_k(\vu) \le\;& a_k \innp{\nabla f(\vx_k)- \vp_k, \vv_k - \vu} \\
    & - a_{k-1}\innp{\nabla f(\vx_{k-1}) - \vp_{k-1}, \vv_{k-1} - \vu} \\  
    & - \frac{1 + A_{k-1}\gamma}{10}\|\vv_k - \vv_{k-1}\|^2 \\
    & + \frac{1 + A_{k-2}\gamma}{10}\|\vv_{k-1} - \vv_{k-2}\|^2.
\end{align*}
\end{restatable}

We are now ready to state the main convergence result of this section.
\begin{restatable}{theorem}{thmacoder} \label{thm:acoder}
Let $\vx_0 \in \dom(g)$ be an arbitrary initial point and consider the updates in Algorithm~\ref{alg:acc-coder}. Then, $\forall k \geq 1$ and any $\vu \in \dom(g)$:
\begin{equation*}
    \bar{f}(\vy_k) - \bar{f}(\vu) +  \frac{3(1+ A_{k-1}\gamma)}{10 A_k}\|\vu - \vv_k\|^2  \leq \frac{\|\vu - \vx_0\|^2}{2A_k}. 
\end{equation*}
In particular, if $\vx^* = \argmin_{\vx} \bar{f}(\vx)$ exists, then
\begin{equation*}
    \bar{f}(\vy_k) - \bar{f}(\vx^*) \leq\frac{\|\vx^* - \vx_0\|^2}{2A_k}. 
\end{equation*}
Further, in this case we also have:
\begin{equation*}
    \|\vv_k - \vx^*\|^2 \leq \frac{5}{3(1 + A_{k-1} \gamma)} \|\vx_0 - \vx^*\|^2,
\end{equation*}
\begin{equation*}
    \|\vy_k - \vx^*\|^2 \leq \pr{\frac{5}{3 A_k} \sum_{i=1}^k \frac{a_i}{1 + A_{i-1} \gamma}} \|\vx_0 - \vx^*\|^2.
\end{equation*}
Finally, in all the bounds we have
\begin{equation*}
    A_k \geq \max \Bigg\{ \frac{2}{5 L}\bigg(1 + \sqrt{\frac{2 \gamma}{5 L}}\bigg)^{k}, \; \frac{k^2}{10 L}\Bigg\}.
\end{equation*}
%
\end{restatable}

\paragraph{Adaptive A-CODER.} The Lipschitz parameter $L$ used in the statement of A-CODER (Algorithm~\ref{alg:acc-coder}) is usually not readily available for typical instances of convex composite minimization problems; however, as we argue in Appendix~\ref{appx:omitted-proofs-acoder}, this parameter can be adaptively estimated using the standard backtracking line search. 
A variant of A-CODER implementing this adaptive estimation of $L$ is provided in Algorithm~\ref{alg:acc-coder-adaptive}. This is enabled by our analysis, which only requires the stated Lipschitz condition to hold between the successive iterates of the algorithm. Notably, unlike randomized algorithms which estimate Lipschitz constants for each of the coordinate blocks (see, e.g.,~\citet{nesterov2012efficiency}), we only need to estimate one summary Lipschitz parameter $L.$

\section{Variance Reduced A-CODER} \label{sec:vr-acoder}
In this section, we assume that the problem \eqref{eq:main-problem} has a finite sum structure, i.e., $f(\vx) = \frac{1}{n} \sum_{t=1}^n f_t(\vx),$ where $n$ may be very large. For this case, we can further reduce the per-iteration cost and improve the complexity results by combining the well-known SVRG-style variance reduction strategy~\cite{johnson2013accelerating} with the results from the previous section to obtain our variance reduced A-CODER (VR-A-CODER). From another perspective, VR-A-CODER can be seen as a cyclic gradient-extrapolated version of the recent VRADA algorithm for finite-sum composite convex minimization~\cite{song2020variance}. 

For this finite-sum setting, we need to make the following stronger assumption for each $f_t(\vx)$.

\begin{algorithm}[t!]
\caption{Variance Reduced A-CODER (Implementable Version)}\label{alg:vr-acc-coder-implementable}
\begin{algorithmic}[1]
\STATE \textbf{Input:} $\vx_0\in\mathrm{dom}(g)$, $\gamma \geq 0$, $L>0$, $m$, $\{\gS^1, \dots, \gS^m\}$
\STATE \textbf{Initialization:} $\vyt_0 = \vv_{1,0} = \vy_{1,0} = \vx_{1, 1} = \vx_{0}$; $\vz_{1, 0} = \mathbf{0}$
\STATE $a_0= A_0 = 0$; $A_1 = a_1 = \frac{1}{4L}$
\STATE $\vz_{1, 1} = \nabla f (\vx_0)$; $\vv_{1,1} = \mathrm{prox}_{a_1 g} (\vx_0 - \vz_{1, 1})$
\STATE $\vyt_1 = \vy_{1, 1} = \vv_{1,1}$
\STATE $\vw_{1, 1, j} = (\vx_{1, 1}^{(1)}, \ldots, \vx_{1, 1}^{(j)}, \vy_{1, 1}^{(j+1)},\ldots, \vy_{1, 1}^{(m)})$
\STATE $\vv_{2,0} = \vv_{1,1}$; $\vw_{2, 0, j} = \vw_{1, 1, j}$; $\vx_{2, 0} = \vx_{1, 1}$; $\vy_{2, 0} = \vy_{1, 1}$; $\vz_{2, 0} = \vz_{1, 1}$
\FOR{$s= 2$ to $S$}
    \STATE $a_s = \sqrt{\frac{K A_{s-1} \pr{1 + A_{s-1} \gamma}}{8 L}}$; $A_s = A_{s-1} + a_s$  
    \STATE $a_{s, 0} = a_{s-1}$; $a_{s,1}=a_{s,2}=\cdots =a_{s,K} = a_s$
    \STATE $\vv_{s, 0} = \vv_{s-1, K}$; $\vw_{s, 0, j} = \vw_{s-1, K, j}$; $\vx_{s, 0} = \vx_{s-1, K}$; $\vy_{s, 0} = \vy_{s-1, K}$; $\vz_{s, 0} = \vz_{s-1, K}$
    \STATE $\vmu_s = \nabla f(\vyt_{s-1})$
    \FOR{$k = 1$ to $K$} 
        \STATE $\vx_{s,k} = \frac{A_{s-1}}{A_s}\vyt_{s-1} + \frac{a_s}{A_s}\vv_{s,k-1}$
        \FOR{$j = m$ to $1$} %
            \STATE $\vw_{s, k, j} = (\vx_{s,k}^{(1)}, \ldots, \vx_{s,k}^{(j)}, \vy_{s,k}^{(j+1)},\ldots, \vy_{s,k}^{(m)})$
            \STATE Choose $t$ in $[n]$ uniformly at random
            \STATE $\tilde{\nabla}_{s,k}^{(j)} = \nabla^{(j)} f_t(\vw_{s, k, j}) - \nabla^{(j)} f_t(\vyt_{s-1}) + \vmu_s^{(j)}$
            \STATE $\vq_{s,k}^{(j)} = \tilde{\nabla}_{s,k}^{(j)} + \frac{a_{s,k-1}}{a_s}(\nabla^{(j)} f_t(\vx_{s,k-1}) - \nabla^{(j)} f_t(\vw_{s, k-1, j}))$
            \STATE $\vz_{s,k}^{(j)} = \vz_{s, k-1}^{(j)} + a_s \vq^{(j)}_{s,k}$
            \STATE $\vv_{s,k}^{(j)} = \mathrm{prox}_{(A_{s-1} + \frac{a_s k}{K}) g^j}(\vx_0^{(j)} - \vz_{s,k}^{(j)} / K)$
            \STATE $\vy_{s,k}^{(j)} = \frac{A_{s-1}}{A_s} \vyt_{s-1}^{(j)} + \frac{a_s}{A_s} \vv_{s,k}^{(j)}$
        \ENDFOR
    \ENDFOR
    \STATE $\vyt_s = \frac{1}{K}\sum_{k=1}^K \vy_{s,k}$
\ENDFOR
\STATE \textbf{return} $\vv_{S, K}, \vyt_S$
\end{algorithmic}	
\end{algorithm}

\begin{assumption} \label{assmpt:average-smooth}
For all $t\in[n],$ $f_t(\vx)$ is convex. Moreover for all $t \in [n]$, there exist positive semidefinite matrices $\{\mQ^1,\mQ^2,\ldots, \mQ^m\}$ such that $\nabla^{(j)} f_t (\cdot)$  is $1$-Lipschitz continuous w.r.t.~the norm $\|\cdot\|_{\mQ^j}$  i.e., $\forall \vx, \vy \in \sR^d, t\in[n]$,
\begin{equation*}
    \|\nabla^{(j)} f_t(\vx) - \nabla^{(j)} f_t(\vy)\|^2   
    \le\; \norm{\vx - \vy}_{\mQ^j}^2.
\end{equation*}
\end{assumption}

\begin{restatable}{lemma}{lemvrsmooth}\label{lem:vr-smooth}
If $f(\vx) = \frac{1}{n}\sum_{t=1}^n f_t(\vx)$ satisfies Assumption \ref{assmpt:average-smooth}, then it   satisfies Assumption \ref{assmpt:Lip-const} and thus Lemma \ref{lemma:Lip-quad-ub} holds.  
\end{restatable}

With this assumption, we can now derive the VR-A-CODER algorithm. Similar to Section \ref{sec:acoder}, we define $\{a_s\}_{s\ge 1}$ and $\{A_s\}_{s\ge 1}$ to be sequences of positive numbers with $A_s =\sum_{i=1}^s a_i , a_0 = A_0 = 0.$
Let $\{\vyt_s\}_{s\ge 0}$  be a sequence of points in $\dom(g)$ which will be determined by the VR-A-CODER algorithm. Our goal here is to bound the function value gap $\bar{f}(\vyt_s) - \bar{f}(\vu)$ above $(\vu\in\dom(g))$. To attain this, we define the estimate sequence $\{\psi_{s,k}\}_{s\ge 1, k\in[K]}$ recursively by $\psi_{1,0}(\vu) = \frac{K}{2}\|\vu - \vx_0\|^2,$
\begin{dmath} \label{eq:es-1} 
    \psi_{1,1}(\vu) = \psi_{1,0}(\vu) + Ka_1(f(\vx_0) + \langle \nabla f(\vx_0), \vu - \vx_0 \rangle + g(\vu)), 
\end{dmath}
and $\psi_{2,0} = \psi_{1,1}$; for $ s \geq 2, 1\le k\le K,$
\begin{dmath} \label{eq:acoder-est-seq-def}
    \psi_{s,k}(\vu) = \psi_{s,k-1}(\vu) + a_s \big( f(\vx_{s,k}) + \innp{\vq_{s,k}, \vu - \vx_{s,k}} + g(\vu)\big),  
\end{dmath}
and $\psi_{s+1,0} = \psi_{s,K}$. In Eqs.~\eqref{eq:es-1}  and \eqref{eq:acoder-est-seq-def}, $\vx_{0}$ is the initial point, $\vx_{s,k}$ and $\vy_{s, k}$ are computed as convex combinations of two points, which is commonly used in Nesterov-style acceleration, and $\vq_{s,k} =(\vq_{s,k}^{(1)}, \vq_{s,k}^{(2)},\ldots, \vq_{s,k}^{(m)})$ is a variance reduced stochastic gradient with extrapolation, which is the main novelty in our algorithm design. Meanwhile, we define $\vv_{s,k}$ by $\vv_{s,k} := \argmin_{\vu\in \sR^d}\psi_{s,k}(\vu)$ and note that due to the specific choice of $\vq_{s,k}$, $\vv_{s,k}$ is updated in a cyclic (block) coordinate way. Furthermore, in VR-A-CODER, for $s\ge 2,$ we define $\vyt_s = \frac{1}{K}\sum_{k=1}^K \vy_{s,k}$.

We start our analysis  by characterizing the gap function, in the following lemma, similar to Lemma~\ref{lemma:a-coder-gap-fn} in Section~\ref{sec:acoder}, although the proof is much more technical in this case. Due to space constraints, we include the full analysis of Variance Reduced A-CODER in Appendix~\ref{appx:omitted-proofs-acoder-vr}.
\begin{restatable}{lemma}{lemmaacodergapfnvr}\label{lemma:a-coder-gap-fn-vr}
For any $\vu \in \sR^d$ and any sequence of vectors $\{\vq_{s,k}\}_{s\geq 2, k\in[K]},$ for all $S\ge 2$, we have
\begin{dgroup*}
\begin{dmath*}
    K A_S(\bar{f}(\vyt_S) - \bar{f}(\vu))
    \end{dmath*}
    \begin{dmath*}
    \le \frac{K}{2}\norm{\vx_0 - \vu}^2  -  \frac{K(1 +  A_S\gamma)}{2}\norm{\vv_{S, K}-\vu}^2 - \frac{K}{4}\norm{\vv_{1,1} - \vv_{1,0}}^2 +  \sum_{s=2}^S \sum_{k=1}^K E_{s,k}(\vu),
\end{dmath*}
\end{dgroup*}
where
\begin{dmath} \label{eq:E-def}
E_{s,k}(\vu) = A_{s}(f(\vy_{s,k}) -  f(\vx_{s,k})) - A_{s-1}(f(\vyt_{s-1}) - f(\vx_{s,k})) +  a_s \innp{\nabla f(\vx_{s, k}) - \vq_{s,k}, \vx_{s, k} - \vu}  + a_s \innp{ \vq_{s,k},  \vx_{s, k} - \vv_{s,k}} - \frac{K(1 + A_{s-1}\gamma)}{2}\|\vv_{s,k} - \vv_{s,k-1}\|^2.
\end{dmath}
\end{restatable}
\begin{figure*}[ht]
    \centering
    {\includegraphics[width=0.32\textwidth]{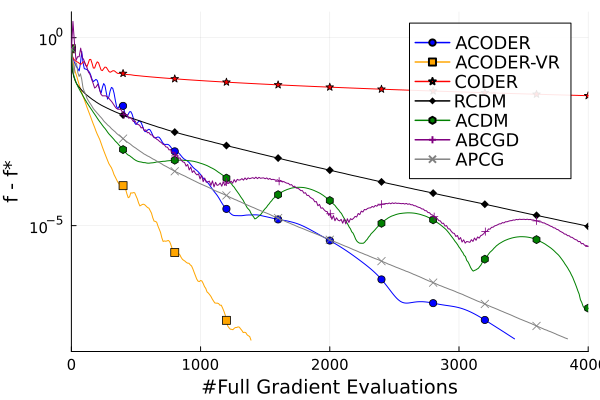}}
    {\includegraphics[width=0.32\textwidth]{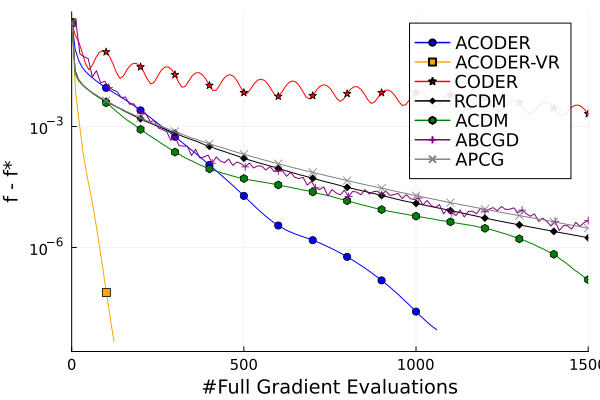}}
    {\includegraphics[width=0.32\textwidth]{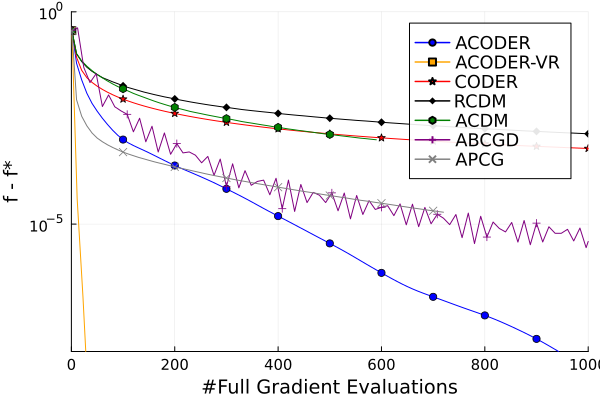}} \\
    {\includegraphics[width=0.32\textwidth]{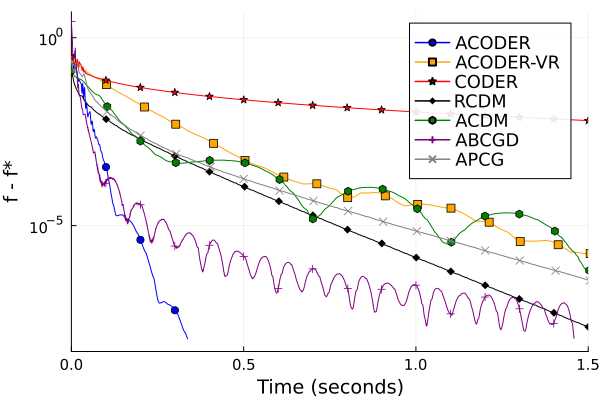}}
    {\includegraphics[width=0.32\textwidth]{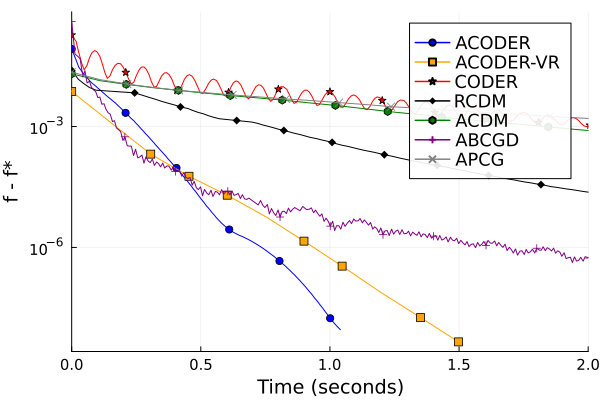}}
    {\includegraphics[width=0.32\textwidth]{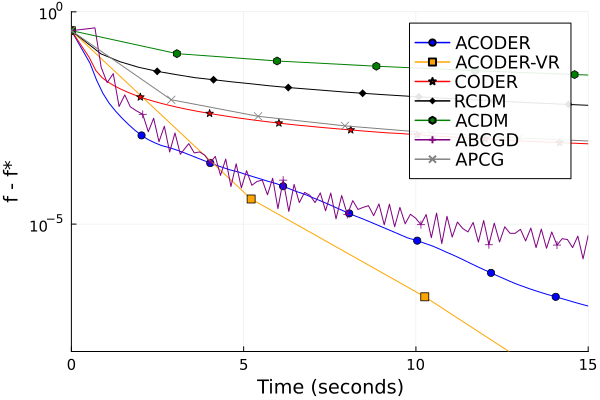}}
    \caption{Performance comparisons between implemented algorithms in terms of the number of full-gradient evaluations and wall-clock time for logistic regression with ridge regularized problems. The top row contains plots against the number of full-gradient evaluations, and the bottom row contains plots against the wall-clock time. The left column is for the sonar dataset, the middle column is for the a1a dataset and the rightmost column is for the a9a dataset, all obtained from LIBSVM~\cite{chang2011libsvm}.}
    \label{fig:numerical-results-0}
\end{figure*}
\begin{figure*}[ht]
    \centering
    {\includegraphics[width=0.32\textwidth]{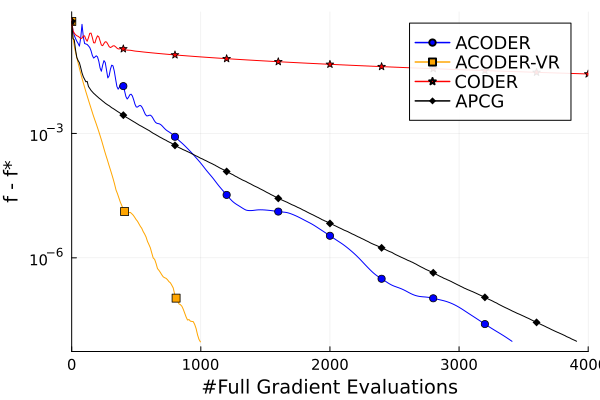}}
    {\includegraphics[width=0.32\textwidth]{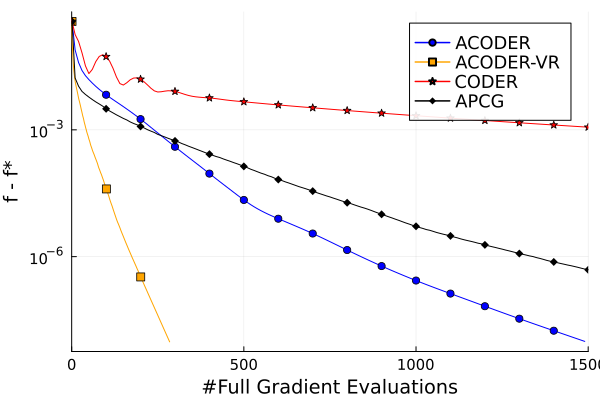}}
    {\includegraphics[width=0.32\textwidth]{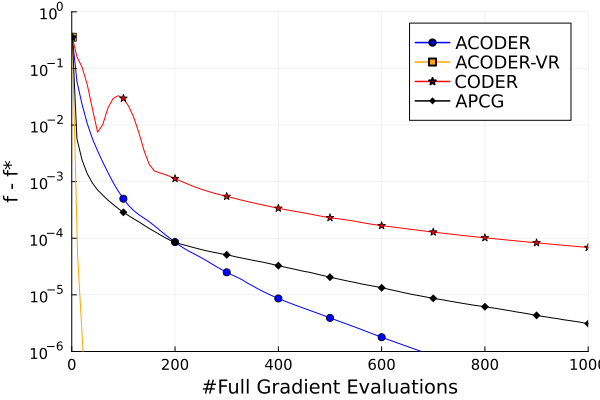}} \\
    {\includegraphics[width=0.32\textwidth]{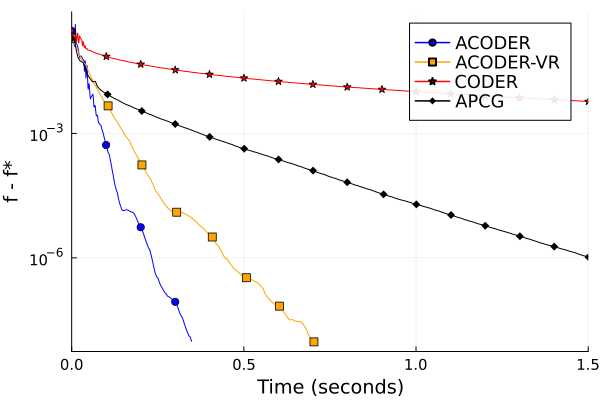}}
    {\includegraphics[width=0.32\textwidth]{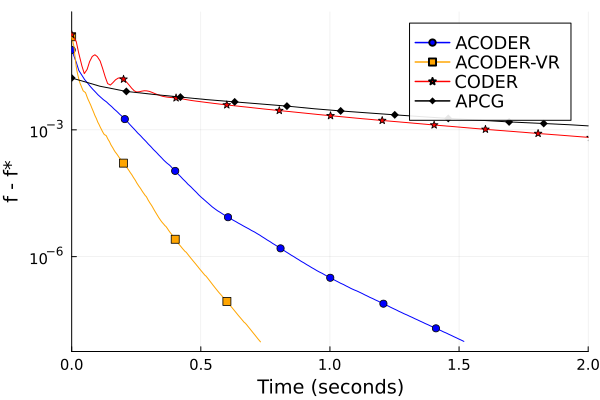}}
    {\includegraphics[width=0.32\textwidth]{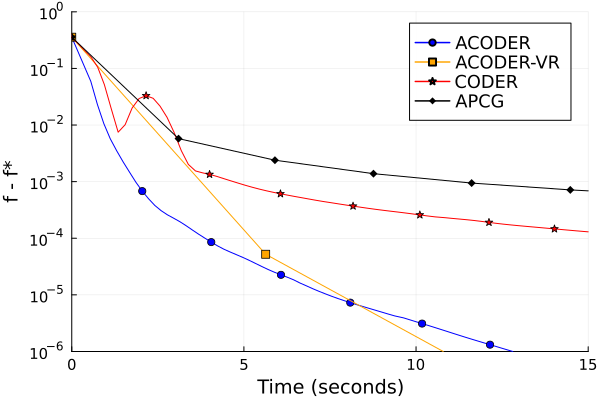}}
    \caption{Performance comparisons between implemented algorithms in terms of the number of full-gradient evaluations and wall-clock time for logistic regression with elastic net regularized problems. The top row contains plots against the number of full-gradient evaluations, and the bottom row contains plots against the wall-clock time. The left column is for the sonar dataset, the middle column is for the a1a dataset and the rightmost column is for the a9a dataset, all obtained from LIBSVM~\cite{chang2011libsvm}.}
    \label{fig:numerical-results-1}
\end{figure*}
\begin{figure*}[ht]
    \centering
    {\includegraphics[width=0.32\textwidth]{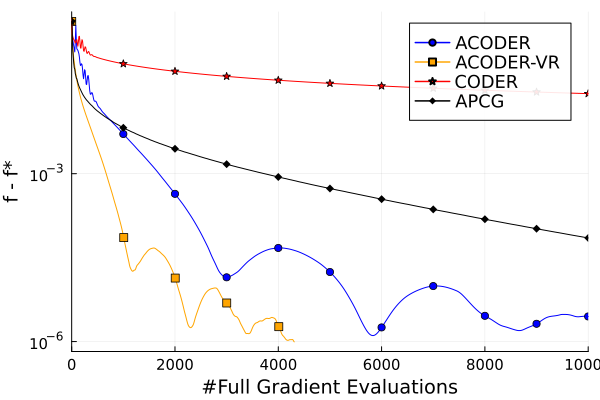}}
    {\includegraphics[width=0.32\textwidth]{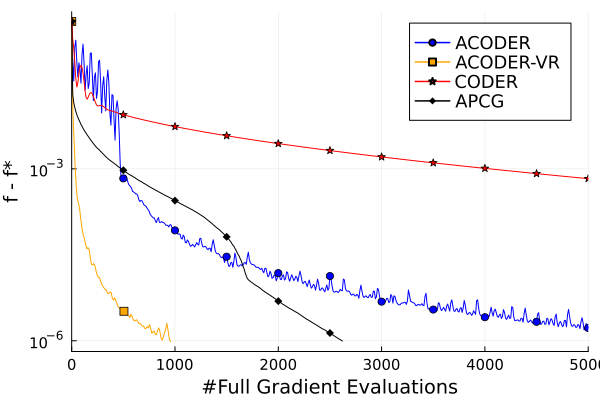}}
    {\includegraphics[width=0.32\textwidth]{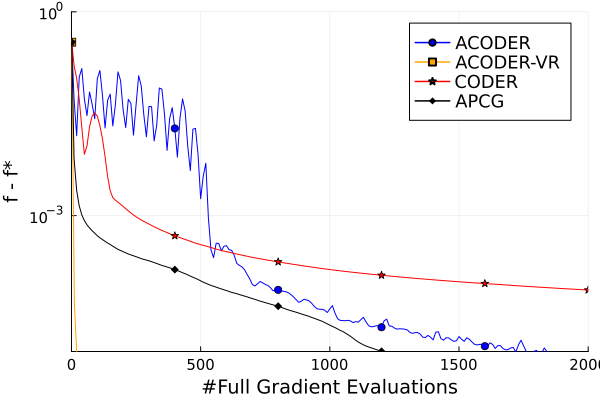}} \\
    {\includegraphics[width=0.32\textwidth]{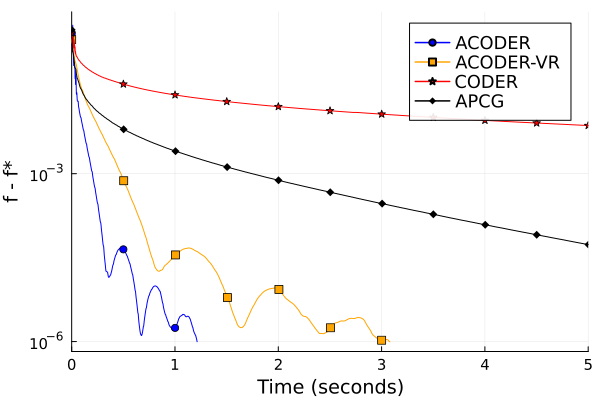}}
    {\includegraphics[width=0.32\textwidth]{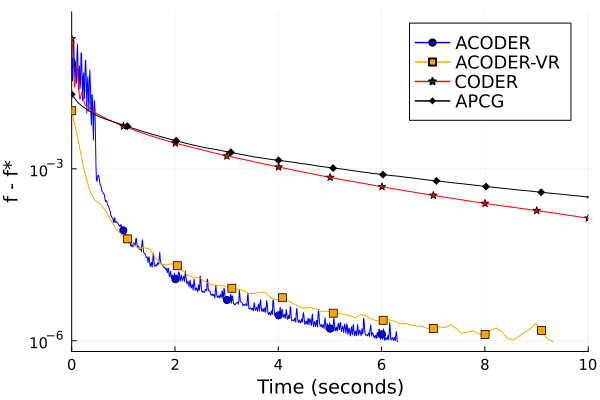}}
    {\includegraphics[width=0.32\textwidth]{figs/a9a-elastic_net-time.png}}
    \caption{Performance comparisons between various algorithms in terms of number of full-gradient evaluations and wall-clock time for logistic regression with LASSO regularized problems. The top row contains plots against the number of full-gradient evaluations, and the bottom two contains plots against wall-clock time. The left column is on sonar dataset, the middle column is on a1a dataset and the rightmost column is on a9a dataset.}
    \label{fig:numerical-results-2}
\end{figure*}

In the following lemma, we bound the expected error terms $\sum_{s=2}^S \sum_{k=1}^K \E \br{E_{s, k} (\vu)}$ arising from the gap bound stated in the previous lemma. This bound is then finally used in Theorem~\ref{thm:vracoder-main} to obtain the claimed convergence results. 

\begin{restatable}{lemma}{lemmavrcodererrortermboundcombined} \label{lemma:vrcoder-error-term-bound-combined}

With $a_s^2 \le \frac{K A_{s-1} \pr{1 + A_{s-1} \gamma}}{8 L}$, $a_{s, k} = a_s$ and $A_{s, k} = A_s$ for $k \in [K]$, $a_{s, 0} = a_{s-1}$ and $A_{s, 0} = A_{s-1}$, then for any fixed $\vu \in \mathrm{dom}(g)$ we have
\begin{dgroup*}
    \begin{dmath*}
        \sum_{s=2}^S \sum_{k=1}^K \E \br{E_{s, k} (\vu)}
    \end{dmath*}
    \begin{dmath*}
        \le - \sum_{j=1}^m a_1 \inner{\nabla^{(j)} f (\vx_{1, 1}) - \nabla^{(j)} f (\vw_{1, 1, j})}{\vv_{1, 1}^{(j)} - \vu^{(j)}} + \sum_{j=1}^m a_{S} \E \br{\inner{\nabla^{(j)} f(\vx_{S, K}) - \nabla^{(j)} f (\vw_{S, K, j})}{\vv_{S, K}^{(j)} - \vu^{(j)}}} + \frac{K}{64} \norm{\vv_{1, 1} - \vv_{1, 0}}^2 - \frac{5 K \pr{1 + A_{S-1} \gamma}}{32} \E\br{\norm{\vv_{S, K} - \vv_{S, K-1}}^2},
    \end{dmath*}
\end{dgroup*}
where $\vx_{1, 1}, \vv_{1, 0} \in \mathrm{dom} (g)$ can be chosen arbitrarily and $\vw_{1, 1, j}$ is defined in Algorithm~\ref{alg:vr-acc-coder-analysis}.
\end{restatable}

Our main result for this section is summarized in the following theorem. 

\begin{restatable}{theorem}{thmvracodermain} \label{thm:vracoder-main}

Let $\vx_0 \in \mathrm{dom}(g)$ be an arbitrary initial point. Fix $K \ge 1$ and consider the updates in Algorithm~\ref{alg:vr-acc-coder-analysis}. Then for $S \ge 2$ and $\forall \vu \in \mathrm{dom}(g)$, we have

\begin{dmath*}
    \E \br{\bar{f} (\vyt_S) - \bar{f} (\vu)} + \frac{9 \pr{1 + A_{S-1} \gamma}}{64 A_S} \E \br{\norm{\vv_{S, K} - \vu}^2} \le \frac{5}{8 A_S} \norm{\vx_0 - \vu}^2.
\end{dmath*}

In particular if $\vx^* = \argmin_{\vx} \bar{f} (\vx)$ exists, then we have 
\begin{equation*}
    \E \br{\bar{f} (\vyt_S) - \bar{f} (\vx^*)} \le \frac{5}{8 A_S} \norm{\vx_0 - \vx^*}^2
\end{equation*}
and 
\begin{equation*}
    \E \br{\norm{\vv_{S, K} - \vx^*}^2} \le \frac{40}{9 \pr{1 + A_{S-1} \gamma}} \norm{\vx_0 - \vx^*}^2.
\end{equation*}

Finally in all the bounds above we have
\begin{equation*}
    A_S \ge \max \set{\frac{S^2 K}{64 L}, \frac{1}{4L} \pr{1 + \sqrt{\frac{K \gamma}{8 L}}}^{S-1}}.
\end{equation*}
\end{restatable}

Note that in Theorem~\ref{thm:vracoder-main}, we can set the number of inner iterations $K$ to be any positive integer. However, in order to balance the computational cost between the outer loop of each epoch and the inner loops, it is optimal to set $K = \Theta (n)$ and for simplicity we can set $K = n$. Therefore, the total number of arithmetic operations required to obtain an $\epsilon$-accurate solution $\vyt_S$ by applying Algorithm~\ref{alg:vr-acc-coder-implementable} such that $\E [\bar{f} (\vyt_S) - \bar{f} (\vx^*)] \le \epsilon$ is at most $O\Big(nd \sqrt{\frac{L \|\vx_0 - \vx^*\|}{n \epsilon}}\Big)$ for the general convex case when $\gamma = 0$, and $O\Big(\frac{n d \log (\epsilon L / \|\vx_0 - \vx^*\|)} {\log (1 + \sqrt{n \gamma / L})}\Big)$ for the strongly convex case when $\gamma > 0$. 

\paragraph{Adaptive VR-A-CODER.} Similar to A-CODER, VR-A-CODER can adaptively estimate the Lipschitz parameter. For completeness, we have included the adaptive version of VR-A-CODER in Algorithm~\ref{alg:vr-acc-coder-adaptive} (Appendix~\ref{appx:omitted-proofs-acoder-vr}). 

\section{Numerical Experiments and Discussion}\label{sec:num-exp}

To verify the effectiveness of our proposed algorithms, we conducted a set of numerical experiments to demonstrate that both A-CODER (Algorithm~\ref{alg:acc-coder}) and VR-A-CODER (Algorithm~\ref{alg:vr-acc-coder-implementable}) almost completely outperform other comparable block-coordinate descent methods in terms of both iteration count and wall-clock time. In particular, we compare against a number of representative methods: CODER \cite{song2021cyclic}, RCDM, ACDM \cite{nesterov2012efficiency}, ABCGD \cite{beck2013convergence} and APCG \cite{lin2015accelerated}. For all the methods, we use the function value gap $f(\vx) - f(\vx^*)$ as the performance measure and we plot our results against the total number of full-gradient evaluations and against wall-clock time in seconds. We implement our experiments in Julia, a high performance programming language designed for numerical analysis and computational science, while optimizing all implementations to the best of our ability. We set the block size to one in all the experiments, i.e., each block corresponds to one coordinate. We discussed in Section~\ref{sec:vr-acoder} that in theory it is optimal to choose $K = \Theta(n)$ in order to balance the computational costs of outer loop and inner loop in VR-A-CODER. We observed in our experiments that it is beneficial to choose $K$ to be slightly smaller than $n$ ($K \approx n/10$) to balance the computational time and the number of full-gradient evaluations.

We consider instances of $\ell_2$-norm (Ridge), $\ell_1$-norm (LASSO) ($\gamma = 0$) and elastic net ($\gamma > 0$) regularized logistic regression problems using three LIBSVM datasets: sonar, a1a and a9a. In the ridge regularized logistic regression problem ( (Figure~\ref{fig:numerical-results-0}), we use $\lambda_2 = 10^{-5}$ for sonar dataset and $\lambda_2 = 10^{-4}$ for a1a and a9a datasets. In the elastic net regularized logistic regression problem (Figure~\ref{fig:numerical-results-1}), we use $\lambda_1 = \lambda_2 = 10^{-5}$ for sonar dataset and $\lambda_1 = \lambda_2 = 10^{-4}$ for a1a and a9a datasets. In the $\ell_1$-norm regularized logistic regression problem (Figure~\ref{fig:numerical-results-2}), we use $\lambda_1 = 10^{-5}$ for sonar dataset and $\lambda_1 = 10^{-4}$ for a1a and a9a datasets. Figures~\ref{fig:numerical-results-1} and~\ref{fig:numerical-results-2}  provide performance comparisons between algorithms considered in terms of the number of full-gradient evaluations and wall-clock time for the elastic net regularized logistic regression problems. We search for the best $L$ or $M$ for each algorithm individually at intervals of $2^{i}$ for $i \in \mathbb{Z}$, and display the best performing runs in the plots. As predicted by our theoretical results, A-CODER and VR-A-CODER exhibit accelerated convergence rates and improved dependence on the number of blocks $m$ even in the worst case, outperforming all other algorithms. In terms of wall-clock time, due to different per-iteration cost of each algorithm in practice, we see a mildly different set of convergence behaviors. However, A-CODER and VR-A-CODER still both perform significantly better than comparable methods.

Combined with the best known theoretical convergence rates guarantee, we believe that this work provides strong supporting arguments for cyclic methods in modern machine learning applications.

\newpage
\bibliographystyle{apalike}
\bibliography{ref.bib}


\newpage
\appendix
\onecolumn

{
\centering \LARGE\bfseries Supplementary Material for Accelerated Cyclic Coordinate Dual Averaging with Extrapolation for Composite Convex Optimization
}

\paragraph{Outline.} 
The appendix of the paper is organized as follows:
\begin{itemize}
    \item Section~\ref{appx:omitted-proofs-acoder} presents the proofs related to the A-CODER algorithm in the main body of the paper, as well as the implementable and adaptive versions of A-CODER.
    \item Section~\ref{appx:omitted-proofs-acoder-vr} presents the proofs related to the A-CODER-VR algorithm in the main body of the paper. We also include the implementable and adaptive versions of A-CODER-VR in this section.
\end{itemize}

\section{Omitted Proofs and Pseudocode for A-CODER} \label{appx:omitted-proofs-acoder}

\lemmaLipquadub*
\begin{proof}
Let $\vz_{j} = (\vx^{(1)}, \ldots, \vx^{(j)}, \vy^{(j+1)},\ldots, \vy^{(m)})$ and observe that $\vz_{m} = \vx$ and $\vz_{0} = \vy$. Then we have
\begin{equation}
    f(\vy) - f(\vx) = \sum_{j=1}^m (f(\vz_{j-1}) - f(\vz_{ j})). 
\end{equation}
As $f$ is continuously differentiable and $\vz_{j}$ and $\vz_{j-1}$ only differ over the $j^{\mathrm{th}}$ block, we further have, by Taylor's theorem,
\begin{dgroup*}
\begin{dmath*}
    f(\vz_{j-1}) - f(\vz_{j}) = \int_0^1 \innp{\nabla f(\vz_{j} + t(\vz_{ j - 1} - \vz_{ j})), \vz_{ j - 1} - \vz_{ j}}\dd t
    = \int_0^1 \innp{\nabla^{(j)} f(\vz_{ j} + t(\vz_{ j - 1} - \vz_{ j})), \vy^{(j)} - \vx^{(j)}}\dd t
\end{dmath*}
\begin{dmath}\label{eq:f-cyc-change-1}
    = \innp{\nabla^{(j)} f(\vx), \vy^{(j)} - \vx^{(j)}} + \int_0^1 \innp{\nabla^{(j)} f(\vz_{j} + t(\vz_{j - 1} - \vz_{j})) - \nabla^{(j)} f(\vx), \vy^{(j)} - \vx^{(j)}}\dd t.
\end{dmath}
\end{dgroup*}
Using Young's inequality, we have, for any $\alpha > 0$,
\begin{dgroup*}
\begin{dmath*}
    \innp{\nabla^{(j)} f(\vz_{ j} + t(\vz_{ j - 1} - \vz_{ j})) - \nabla^{(j)} f(\vx), \vy^{(j)} - \vx^{(j)}}
\end{dmath*}
\begin{dmath*}    
    \leq  \frac{\alpha}{2}\|\nabla^{(j)} f(\vz_{ j} + t(\vz_{ j - 1} - \vz_{ j})) - \nabla^{(j)} f(\vx)\|^2 + \frac{1}{2\alpha}\|\vy^{(j)} - \vx^{(j)}\|^2
\end{dmath*}
\begin{dmath*}
    \leq \frac{\alpha}{2}\|\vz_{ j} + t(\vz_{ j - 1} - \vz_{ j}) - \vx\|_{\mQ^j}^2  + \frac{1}{2\alpha}\|\vy^{(j)} - \vx^{(j)}\|^2
\end{dmath*}
\begin{dmath*}
    \leq \frac{\alpha}{2}\Big[(1-t)\|\vz_{ j} - \vx\|_{\mQ^j}^2 + t\|\vz_{j-1} - \vx\|_{\mQ^j}^2\Big] + \frac{1}{2\alpha}\|\vy^{(j)} - \vx^{(j)}\|^2,
\end{dmath*}
\end{dgroup*}
where the second inequality is by our block Lipschitz assumption from Assumption~\ref{assmpt:Lip-const} and the last line is by Jensen's inequality. Now  observe that $\vz_{j}$ and $\vx$ agree on the first $j$ blocks.
Thus, we can write $\vz_j - \vx = (\vy - \vx)_{\ge j+1}$ and $\vz_{j-1} - \vx = (\vy - \vx)_{\ge j}$, while noting that we have $\norm{(\vy - \vx)_{\ge j}}_{Q_j}^2 = \norm{\vy - \vx}_{(Q_j)_{\ge j}}^2$ and $\norm{(\vy - \vx)_{\ge j+1}}_{Q_j}^2 = \norm{\vy - \vx}_{(Q_j)_{\ge j+1}}^2$. So by combining with Eq.~\eqref{eq:f-cyc-change-1} and integrating over $t$, we have, $\forall \alpha > 0$,
\begin{equation}\label{eq:f-cyc-change-2}
    \begin{aligned}
        f(\vz_{ j-1}) - f(\vz_{ j}) & \leq \innp{\nabla^{(j)} f(\vx), \vy^{(j)} - \vx^{(j)}} + \frac{1}{2\alpha}\|\vy^{(j)} - \vx^{(j)}\|^2\\
        & \quad + \frac{\alpha}{4}\Big(\|\vy - \vx\|_{(\mQ^j)_{\ge j}}^2 + \norm{\vy - \vx}_{(\mQ^j)_{\ge j+1}}^2\Big).
    \end{aligned}
\end{equation}
Summing Eq.~\eqref{eq:f-cyc-change-2} over $j \in [m]$ and using the definition of Mahalanobis norm, we finally get
\begin{align*}
    f(\vy) - f(\vx) =\; & \sum_{j=1}^m (f(\vz_{ j-1}) - f(\vz_{ j}))\\
    \leq \; & \innp{\nabla f(\vx), \vy - \vx} + \frac{\alpha}{4} (\vy - \vx)^T \tilde{\mQ} (\vy - \vx) + \frac{1}{2\alpha}\|\vy - \vx\|^2 \\
    \leq \; & \innp{\nabla f(\vx), \vy - \vx} + \pr{\frac{1}{2\alpha} + \frac{\alpha L^2}{8}} \norm{\vy - \vx}^2,
\end{align*}
where we used Holder's inequality and the definition of $L$ in Assumption~\ref{assmpt:Lip-const}. Letting $\alpha = \frac{2}{L}$ now completes the proof of the first part.

The second part of the proof is standard and is provided for completeness. Let $\vx, \vy$ be any two points from $\sR^d.$ Define $h_{\vx}(\vy) = f(\vy) - \innp{\nabla f(\vx), \vy}.$ Observe that $h_{\vx}(\vy)$ is convex (as the sum of a convex function $f(\vy)$ and a linear function $- \innp{\nabla f(\vx), \vy}$) and is minimized at $\vy = \vx$ (as for any $\vy \in \sR^d,$ $h_{\vx}(\vy) - h_{\vx}(\vx) = f(\vy) - f(\vx) - \innp{\nabla f(\vx), \vy - \vx} \geq 0,$ by convexity of $f$). Observe further that for any $\vy, \vz \in \sR^d$, we have
\begin{align*}
    h_{\vx}(\vy) - h_{\vx}(\vz) - \innp{\nabla h_{\vx}(\vz), \vy - \vz}
   =&\; f(\vy) - f(\vz) - \innp{\nabla f(\vz), \vy - \vz}\\
   \leq &\; \frac{L}{2} \|\vy - \vz\|^2,
\end{align*}
where the last inequality is by the first part of the proof.
The last inequality and the fact that $\vx$ minimizes $h_{\vx}$ now allow us to conclude that
\begin{align*}
    h_{\vx}(\vx) &\leq h_{\vx}\Big(\vy - \frac{1}{L}\nabla h_{\vx}(\vy)\Big)\\
    &\leq h_{\vx}(\vy) - \frac{1}{2 L}\|\nabla h_{\vx}(\vy)\|^2.
\end{align*}
To complete the proof, it remains to plug the definition of $h_x(\cdot)$ into the last inequality, and rearrange.
\end{proof}

\lemmaacodergapfn*

\begin{proof}
As $\vy_k = \frac{1}{A_k}\sum_{i=1}^k a_i \vv_i$ and $g$ is convex, we have $g(\vy_k)\le\frac{1}{A_k}\sum_{i=1}^k a_i g(\vv_i)$ and thus,
\begin{align}
    A_k \bar{f}(\vy_k) \le\;& A_k  f(\vy_k) + \sum_{i=1}^k a_i g(\vv_i) \nonumber\\ 
    =\;& \sum_{i=1}^k (A_i  f(\vy_i) - A_{i-1}  f(\vy_{i-1}))+ \sum_{i=1}^k a_i g(\vv_i),   \label{eq:upper}
\end{align}
where the equality is by $A_0 = 0$.
Then, as $f$ is convex and $\bar{f} = f + g,$ we have, $\forall \vu,$
\begin{align*}
    A_k   \bar{f}(\vu) =\; & \sum_{i=1}^k a_i \bar{f}(\vu)  \geq \sum_{i=1}^k a_i \big(f(\vx_i) + \innp{\nabla f(\vx_i), \vu - \vx_i} + g(\vu) \big)\\
    =\; & \sum_{i=1}^k a_i \big(f(\vx_i) + \innp{\vq_i, \vu - \vx_i} + g(\vu) \big)+ \sum_{i=1}^k a_i \innp{\nabla f(\vx_i) -\vq_i, \vu - \vx_i} \\
    =\; & \psi_k(\vu) - \psi_0(\vu) + \sum_{i=1}^k a_i \innp{\nabla f(\vx_i) -\vq_i, \vu - \vx_i} \\
    \ge \; &\psi_k(\vv_k) +  \frac{1 + A_k\gamma}{2}\|\vu - \vv_k\|^2 - \frac{1}{2}\|\vu - \vx_0\|^2  \\
     &   + \sum_{i=1}^k a_i \innp{\nabla f(\vx_i) -\vq_i, \vu - \vx_i},
\end{align*}
where the first inequality is by the convexity of $f$, the third equality is by the recursive definition of $\psi_k(\vu)$, and the last inequality is by the $(1 + A_k\gamma)$-strong convexity of $\psi_k(\vu)$, $\vv_k = \argmin_{\vu}\psi_k(\vu)$ which implies $\psi_k(\vu) \geq \psi_k(\vv_k) +  \frac{1 + A_k\gamma}{2}\|\vu - \vv_k\|^2$, and the definition of $\psi_0(\vu)$.  

Then as $\psi_0(\vv_0) = 0,$ using the recursive definition of $\psi_k,$ we have
\begin{dgroup*}
\begin{dmath*}
    \psi_k(\vv_k) = \sum_{i=1}^k (\psi_i(\vv_i) - \psi_{i-1}(\vv_{i-1}))
\end{dmath*}
\begin{dmath*}
    = \sum_{i=1}^k\Big( (\psi_{i-1}(\vv_i) - \psi_{i-1}(\vv_{i-1})) + a_i\big(f(\vx_i) + \innp{ \vq_i, \vv_i - \vx_i} + g(\vv_i) \big)\Big)
\end{dmath*}
\begin{dmath} \label{eq:lower2}
    \ge \sum_{i=1}^k \Big( \frac{1+A_{i-1}\gamma}{2}\|\vv_i - \vv_{i-1}\|^2 + a_i(f(\vx_i) + \innp{\vq_i, \vv_i - \vx_i} + g(\vv_i))\Big),
\end{dmath}
\end{dgroup*}
where the last inequality is by the $(1+A_{i-1}\gamma)$-strong convexity of $\psi_{i-1}$ and the optimality of $\vv_{i-1}$. 
Combining Eqs.~\eqref{eq:upper} and \eqref{eq:lower2}, we have %
\begin{align}
A_k(\bar{f}(\vy_k) -  \bar{f}(\vu)) \le \sum_{i=1}^k E_i(\vu) - \frac{1 + A_k\gamma}{2}\|\vu - \vv_k\|^2  + \frac{1}{2}\|\vu - \vx_0\|^2,
\end{align}
where $E_i(\vu)$ is defined in \eqref{eq:acc-E_i}. 
\end{proof}

\lemmaacodergapchange*

\begin{proof}
By the convexity of $f$, we have $$ f(\vy_{k-1}) - f(\vx_k) \geq \innp{\nabla f(\vx_k),  \vy_{k-1} - \vx_k}.$$ Then by applying Lemma~\ref{lemma:Lip-quad-ub}, we have
\begin{dgroup*}
    \begin{dmath*}
        A_k(f(\vy_k) - f(\vx_k)) -  A_{k-1}(f(\vy_{k-1})-f(\vx_k))
        \leq \innp{\nabla f(\vx_k), A_k \vy_k - A_{k-1}\vy_{k-1} - a_k \vx_k} + \frac{A_k L}{2}\|\vy_k - \vx_k\|^2
    \end{dmath*}
    \begin{dmath} \label{eq:acoder-change-in-fn}
        = a_k \innp{\nabla f(\vx_k), \vv_k - \vx_k} + \frac{{a_k}^2 L}{2 A_k} \|\vv_k - \vv_{k-1}\|^2,
    \end{dmath}
\end{dgroup*}
where we used the definitions of $\vy_k$ and $\vx_k$ from Algorithm~\ref{alg:acc-coder} in the last equality. 
Combining Eq.~\eqref{eq:acc-E_i} (with $i=k$) in Lemma~\ref{lemma:a-coder-gap-fn} and Eq.~\eqref{eq:acoder-change-in-fn}, we have
\begin{equation}\label{eq:acoder-change-in-gap-2}
    \begin{aligned}
        E_k(\vu) \le\; & \Big(\frac{{a_k}^2 L}{2 A_k} - \frac{1 + A_{k-1}\gamma}{2}\Big)\|\vv_k - \vv_{k-1}\|^2 
        + a_k \innp{\nabla f(\vx_k) - \vq_k, \vv_k - \vu}. 
\end{aligned}
\end{equation}
Thus by rewriting the second term as the sum of inner products over the $m$ blocks and by using the definition of $\vq_k^{(j)}$ in Algorithm~\ref{alg:acc-coder}, we have
\begin{dgroup*}
    \begin{dmath*}
    a_k \langle\nabla f(\vx_k) - \vq_k, \vv_k - \vu\rangle = a_k \sum_{j=1}^m \innp{\nabla^{(j)} f(\vx_k) - \vq_k^{(j)}, \vv_k^{(j)} - \vu^{(j)}}
    \end{dmath*}
    \begin{dmath} \label{eq:lemma-innerproduct-sum-1}
        = \sum_{j=1}^m \left[ a_k \innp{\nabla^{(j)} f(\vx_k) - \vp_k^{(j)}, \vv_k^{(j)} - \vu^{(j)}} - a_{k-1} \innp{\nabla^{(j)} f(\vx_{k-1}) - \vp_{k-1}^{(j)}, \vv_{k-1}^{(j)} - \vu^{(j)}}\right] + a_{k-1} \sum_{j=1}^m \innp{\nabla^{(j)} f(\vx_{k-1}) - \vp_{k-1}^{(j)}, \vv_{k-1}^{(j)} - \vv_{k}^{(j)}}.
    \end{dmath}
\end{dgroup*}
Notice that the first two inner product terms in the first line of Eq.~\eqref{eq:lemma-innerproduct-sum-1} telescope when summed over $k$, therefore it remains to bound $a_{k-1} \sum_{j=1}^m\innp{\nabla^{(j)} f(\vx_{k-1}) - \vp_{k-1}^{(j)}, \vv_{k}^{(j)} - \vv_{k-1}^{(j)}}$. In particular we let $\vw_{k, j} = (\vx^{1}_{k}, \ldots, \vx^{j}_{k}, \vy^{j+1}_{k},\ldots,   \vy^{m}_{k})$ so that $\vp_k^{(j)} = \nabla^{(j)} f(\vw_{k, j})$, then we have
\begin{dgroup*}
    \begin{dmath*}
        \innp{\nabla^{(j)} f (\vx_{k-1}) - \vp_{k-1}^{(j)}, \vv_{k-1}^{(j)} - \vv_k^{(j)}}
        = \innp{\nabla^{(j)} f (\vx_{k-1}) - \nabla^{(j)} f(\vw_{k-1, j}), \vv_{k-1}^{(j)} - \vv_k^{(j)}}
    \end{dmath*}
    \begin{dmath*}
        \le \frac{\alpha}{2} \norm{\nabla^{(j)} f(\vx_{k-1}) - \nabla^{(j)} f(\vw_{k-1, j})}^2 + \frac{1}{2 \alpha} \norm{\vv_{k-1}^{(j)} - \vv_k^{(j)}}^2
    \end{dmath*}
    \begin{dmath} \label{eq:lemma-innerproduct-sum-2}
        \le \frac{\alpha}{2} \norm{\vx_{k-1} - \vw_{k-1, j}}_{\mQ^j}^2 + \frac{1}{2 \alpha} \norm{\vv_{k-1}^{(j)} - \vv_k^{(j)}}^2
    \end{dmath}
\end{dgroup*}
where the first inequality holds for any $\alpha > 0$ by Young's inequality and the second inequality is by Assumption~\ref{assmpt:Lip-const}. Notice that $\vx_{k-1}$ and $\vw_{k-1, j}$ agree on the first $j$ blocks, so similar to the proof of Lemma~\ref{lemma:Lip-quad-ub} we can write $\vx_{k-1} - \vw_{k-1, j} = (\vy_{k-1} - \vx_{k-1})_{\ge j+1}$. Therefore by applying similar arguments as Lemma~\ref{lemma:Lip-quad-ub}, we get
\begin{dgroup*} 
    \begin{dmath*}
        \sum_{j=1}^m \norm{\vx_{k-1} - \vw_{k-1, j}}_{\mQ^j}^2 = \sum_{j=1}^m \norm{\vy_{k-1} - \vx_{k-1}}_{(\mQ^j)_{\ge j+1}}^2
    \end{dmath*}
    \begin{dmath*}
        \le \sum_{j=1}^m \norm{\vy_{k-1} - \vx_{k-1}}_{(\mQ^j)_{\ge j+1}}^2 + \sum_{j=1}^m \norm{\vy_{k-1} - \vx_{k-1}}_{(\mQ^j)_{\ge j}}^2 
    \end{dmath*}
    \begin{dmath*}
        = \norm{\vy_{k-1} - \vx_{k-1}}_{\tilde{\mQ}}^2
    \end{dmath*}
    \begin{dmath} \label{eq:lemma-innerproduct-sum-3}
        \le \frac{a_{k-1}^2 L^2}{2 A_{k-1}^2} \norm{\vv_{k-1} - \vv_{k-2}}^2
    \end{dmath}
\end{dgroup*}
where we used the non-negativity of Mahalanobis norm w.r.t. semi-positive definite matrix in the first inequality and the definition of $\vx_k$, $\vy_k$ and $L$ in the last inequality. Lastly, by combining Eqs.~\eqref{eq:acoder-change-in-gap-2}, \eqref{eq:lemma-innerproduct-sum-1}, \eqref{eq:lemma-innerproduct-sum-2} and \eqref{eq:lemma-innerproduct-sum-3}, we have
\begin{align*}
    E_k (\vu) \le \; & a_k \innp{\nabla f(\vx_k) - \vp_k, \vv_k - \vu} - a_{k-1}\innp{\nabla f(\vx_{k-1}) - \vp_{k-1}, \vv_{k-1} - \vu} \\
    & + \pr{\frac{a_k^2 L}{2 A_k} - \frac{1 + A_{k-1}\gamma}{2} + \frac{a_{k-1}}{2 \alpha}} \norm{\vv_k - \vv_{k-1}}^2 + \pr{\frac{\alpha a_{k-1}^3 L^2}{4 A_{k-1}^2}} \norm{\vv_{k-1} - \vv_{k-2}}^2.
\end{align*}
It remains to choose $\alpha = \frac{A_{k-1}}{a_{k-1} L}$ and some sequence $\set{a_i}_i^k$ such that $ \frac{a_k^2}{A_k} \le \frac{2\pr{1 + A_{k-1}\gamma}}{5 L}$.
\end{proof}

\thmacoder*

\begin{proof}
By Lemma \ref{lemma:acoder-gap-change}, and using the fact $A_0 = a_0 = 0$ and $\vv_0 = \vv_{-1}$, we have   
\begin{equation} \label{eq:acoder-fin-gap-bnd-1}
    \sum_{i=1}^k E_i(\vu) \leq - \frac{1 + A_{k-1}\gamma}{10}\|\vv_k - \vv_{k-1}\|^2  + a_k \innp{\nabla f(\vx_k) - \vp_k, \vv_k - \vu}.
\end{equation}
Same as in the proof of Lemma~\ref{lemma:acoder-gap-change}, we can bound $a_k \innp{\nabla f(\vx_k) - \vp_k, \vv_k - \vu}$ using Young's inequality and the definition of smoothness for $f.$ In particular, for any $\alpha > 0,$
\begin{equation}\notag
\begin{aligned}
    a_k \innp{\nabla f(\vx_k) - \vp_k, \vv_k - \vu} \leq \; & a_k\Big(\frac{\alpha L^2}{4}\|\vy_k - \vx_k\|^2 + \frac{1}{2\alpha}\|\vu - \vv_k\|^2\Big)\\
    =\; & a_k\Big(\frac{\alpha L^2 {a_k}^2}{4{A_k}^2}\|\vv_k - \vv_{k-1}\|^2 + \frac{1}{2\alpha}\|\vu - \vv_k\|^2\Big).
\end{aligned}
\end{equation}
Choosing $\alpha = \frac{A_k}{a_k L}$ and using $\frac{{a_k}^2}{A_k} \le \frac{2 \pr{1 + A_{k-1}\gamma}}{5 L}$ , we get
\begin{equation} \label{eq:acoder-inn-prod-term}
    a_k \innp{\nabla f(\vx_k) - \vp_k, \vv_k - \vu} \leq \frac{\pr{1 + A_{k-1}\gamma}}{10} \|\vv_k - \vv_{k-1}\|^2 + \frac{\pr{1 + A_{k-1}\gamma}}{5} \|\vu - \vv_k\|^2.
\end{equation}
Then combining Lemma \ref{lemma:a-coder-gap-fn}, Eq.~\eqref{eq:acoder-inn-prod-term} and Eq.~\eqref{eq:acoder-fin-gap-bnd-1} with the fact $A_{k-1}\le A_k$, we have
\begin{align}
    (\bar{f}(\vy_k) - \bar{f}(\vu))   + \frac{3 (1+ A_{k-1}\gamma)}{10 A_k}\|\vu - \vv_k\|^2  \le \frac{1}{2A_k}\|\vu - \vx_0\|^2. \label{eq:acc-bound}
\end{align}

Assume now that $\vx^* = \argmin_{\vx}\bar{f}(\vx)$ exists. As $\bar{f}(\vy_k) - \bar{f}(\vx^*)\geq 0,$ Eq.~\eqref{eq:acc-bound} implies
\begin{equation}\label{eq:vv_k-dist-bnd}
    \|\vv_k - \vx^*\|^2 \leq \frac{5}{3(1 + A_{k-1} \gamma)} \|\vx_0 - \vx^*\|^2.
\end{equation}
Using Jensen's inequality, as $\vy_k = \frac{1}{A_k}\sum_{i=1}^k a_i \vv_i,$ we also have from Eq.~\eqref{eq:vv_k-dist-bnd}
\begin{align*}
    \|\vy_k - \vx^*\|^2 \leq \pr{\frac{5}{3 A_k} \sum_{i=1}^k \frac{a_i}{1 + A_{i-1} \gamma}} \|\vx_0 - \vx^*\|^2.
\end{align*}

Finally, recall once again that $\{a_k\}_{k\geq 1}$ is chosen so that $\frac{{a_k}^2}{A_k} = \frac{2(1 + A_{k-1}\gamma)}{5 L}.$ When $\gamma = 0,$ this leads to the standard $A_k \geq \frac{k^2}{10 L}$ growth of accelerated algorithms by choosing $a_k = \frac{k}{5 L}$ for $k \geq 1.$ When $\gamma > 0,$ we have $\frac{a_k}{A_{k-1}} > \sqrt{\frac{2 \gamma}{5 L}},$ and it remains to use that $A_k = \frac{A_k}{A_{k-1}}\cdot \dots \cdot \frac{A_2}{A_1}\cdot A_1 = A_1 \Big(1 + \sqrt{\frac{2 \gamma}{5 L}}\Big)^{k-1}$ where $a_1 = A_1 = \frac{2}{5 L}$ using the choice of $a_k$ in Algorithm~\ref{alg:acc-coder} and $A_0 = a_0 = 0$, completing the proof.
\end{proof}

\subsection{(Lipschitz) Parameter-Free  A-CODER}\label{appx-subsection:acoder-adaptive}

Similar to CODER, it is possible to adaptively estimate the Lipschitz parameter $L_k$ for A-CODER. Note that in the case of A-CODER, all that is needed for the analysis from Section~\ref{sec:acoder} to apply is that the quadratic bound from Lemma~\ref{lemma:Lip-quad-ub} holds between $\vx_k$ and $\vy_k.$ A variant of A-CODER that implements this adaptive estimation is provided in Algorithm~\ref{alg:acc-coder-adaptive}.
%
%
\begin{algorithm*}[t!]
\caption{Adaptive Accelerated Cyclic cOordinate Dual avEraging with extRapolation (Ada-A-CODER)} \label{alg:acc-coder-adaptive}
\begin{algorithmic}[1]
\STATE \textbf{Input:} $\vx_0\in\mathrm{dom}(g), \gamma \geq 0, L_0 > 0, m, \{S^1, \dots, S^m\}$
\STATE \textbf{Initialization:} $\vx_{-1} = \vx_{0} = \vv_{-1} = \vv_0 = \vy_0, \vp_0 = \nabla f(\vx_0), \vz_0 = \vzero$, $a_0= A_0 = 0$
\FOR {$k = 1$ to $K$} 
    \STATE $L_k = L_{k-1} / 2$
    \REPEAT
        \STATE $L_k = 2 L_k$
        \STATE Set $a_k > 0$ be largest value s.t. $\frac{a_k^2}{A_{k}} \le \frac{2 \pr{1 + A_{k-1} \gamma}}{5 L_k}$ where $A_k = A_{k-1} + a_k$
        \STATE $\vx_k = \frac{A_{k-1}}{A_k}\vy_{k-1} + \frac{a_k}{A_k}\vv_{k-1}$
        \FOR {$j = m$ to $1$}
            \STATE $\vp_k^{(j)} = \nabla^{(j)} f(\vx^{(1)}_{k}, \ldots, \vx^{(j)}_{k}, \vy^{(j+1)}_{k},\ldots,   \vy^{(m)}_{k})$
            \STATE $\vq^{(j)}_k = \vp_k^{(j)} + \frac{a_{k-1}}{a_k}(\nabla^{(j)}f(\vx_{k-1}) - \vp_{k-1}^{(j)})$
            \STATE $\vz_k^{(j)} = \vz_{k-1}^{(j)} + a_k\vq^{(j)}_k$
            \STATE $\vv_{k}^{(j)} = \mathrm{prox}_{A_k g^j}(\vx_0^{(j)} - \vz_k^{(j)})$
            \STATE $\vy_k^{(j)} = \frac{A_{k-1}}{A_k}\vy_{k-1}^{(j)} + \frac{a_k}{A_k}\vv_k^{(j)}$
        \ENDFOR
    \UNTIL{$f(\vy_k) \le f(\vx_k) + \inner{\nabla f(\vx_k)}{\vy_k - \vx_k} + \frac{L_k}{2} \norm{\vy_k - \vx_k}^2$}
\ENDFOR
\STATE \textbf{return} $\vv_K, \vy_K$
\end{algorithmic}	
\end{algorithm*}
%
%

\section{Omitted Proofs and Pseudocode for ACODER-VR} \label{appx:omitted-proofs-acoder-vr}

\begin{algorithm*}[t!]
\caption{Variance Reduced A-CODER (Analysis Version)}\label{alg:vr-acc-coder-analysis}
\begin{algorithmic}[1]
\STATE \textbf{Input:} $\vx_0\in\mathrm{dom}(g), \gamma \geq 0, L>0, m, \{\gS^1, \dots, \gS^m\}$
\STATE \textbf{Initialization:} $\vyt_0 = \vv_{1,0} = \vy_{1,0} = \vx_{1, 1} = \vx_{0}$
\STATE $a_0= A_0 = 0$; $A_1 = a_1 = \frac{1}{4L}$
\STATE $\psi_{1,0}(\cdot) = \frac{K}{2}\|\cdot - \vx_0\|^2$ 
\STATE $\vv_{1,1} = \argmin_{\vv}\{\psi_{1,1}(\vv) := \psi_{1,0}(\vv) + K a_1(f(\vx_0) + \langle \nabla f(\vx_0), \vv - \vx_0 \rangle + g(\vv))\}$
\STATE $\vw_{1, 1, j} = (\vx_{1, 1}^{(1)}, \ldots, \vx^{(j)}_{1, 1}, \vy^{(j+1})_{1, 1},\ldots, \vy^{(m)}_{1, 1})$
\STATE $\vyt_1 = \vv_{2,0} = \vy_{1, 1} = \vv_{1,1}$; $\vw_{2, 0, j} = \vw_{1, 1, j}$; $\psi_{2,0} = \psi_{1,1}$
\FOR{$s= 2$ to $S$}
    \STATE Set $a_s > 0$ s.t. $a_s^2 = \frac{K A_{s-1} \pr{1 + A_{s-1} \gamma}}{8 L}$; $A_s = A_{s-1} + a_s$  
    \STATE $a_{s, 0} = a_{s-1}$; $a_{s,1}=a_{s,2}=\cdots =a_{s,K} = a_s$
    \STATE $\vx_{s, 0} = \vx_{s-1, K}$; $\vy_{s, 0} = \vx_{s-1, K}$; $\vw_{s, 0, j} = \vw_{s-1, K, j}$; $\vv_{s, 0} = \vv_{s-1, K}$; $\psi_{s, 0} = \psi_{s-1, K}$
    \STATE $\vmu_s = \nabla f(\vyt_{s-1})$
    \FOR{$k = 1$ to $K$} 
        \STATE $\vx_{s,k} = \frac{A_{s-1}}{A_s}\vyt_{s-1} + \frac{a_s}{A_s}\vv_{s,k-1}$
        \FOR{$j = m$ to $1$} %
            \STATE $\vw_{s, k, j} = (\vx^{(1)}_{s,k}, \ldots, \vx^{(j)}_{s,k}, \vy^{(j+1)}_{s,k},\ldots, \vy^{(m)}_{s,k})$
            \STATE Choose $t$ in $[n]$ uniformly at random
            \STATE $\tilde{\nabla}_{s,k}^{(j)} = \nabla^{(j)} f_t(\vw_{s, k, j}) - \nabla^{(j)} f_t(\vyt_{s-1}) + \vmu_s^{(j)}$
            \STATE $\vq^{(j)}_{s,k} = \tilde{\nabla}_{s,k}^{(j)} + \frac{a_{s,k-1}}{a_s}(\nabla^{(j)} f_t(\vx_{s,k-1}) - \nabla^{(j)} f_t(\vw_{s, k-1, j}))$
            \STATE $\vv_{s,k}^{(j)} = \argmin_{\vv^{(j)} \in \sR^{d^j}}\{\psi^j_{s,k}(\vv^{(j)}) := \psi^j_{s,k-1}(\vv^j) + a_s (  \frac{1}{m}f(\vx_{s,k}) + \langle \vq^{(j)}_{s,k}, \vv^{(j)} - \vy_{s,k-1}^{(j)} \rangle  + g^j (\vv^{(j)}))\}$
            \STATE $\vy_{s,k}^{(j)} = \frac{A_{s-1}}{A_s} \vyt_{s-1}^{(j)} + \frac{a_s}{A_s}\vv_{s,k}^{(j)}$
        \ENDFOR
    \ENDFOR
    \STATE $\vyt_s = \frac{1}{K}\sum_{k=1}^K \vy_{s,k}$
\ENDFOR
\STATE \textbf{return} $\vv_{S, L}, \vyt_S$
\end{algorithmic}	
\end{algorithm*}

\lemvrsmooth*

\begin{proof}
By using Jensen's inequality and Assumption~\ref{assmpt:average-smooth}, we have
\begin{equation*}
    \norm{\nabla^{(j)} f(\vx) - \nabla^{(j)} f(\vy)}^2
    \le \frac{1}{n} \sum_{t=1}^n \norm{\nabla^{(j)} f_t(\vx) - \nabla^{(j)} f_t(\vy)}^2
    \le \norm{\vx - \vy}_{\mQ^j}^2.
\end{equation*}
\end{proof}

\lemmaacodergapfnvr*

\begin{proof}

As $f$ is convex and $\bar{f} = f + g,$ we have: $\forall \vu$,
\begin{dgroup*}
    \begin{dmath*}
        K A_S\bar{f}(\vu) = \sum_{s=1}^S \sum_{k=1}^K a_s \bar{f}(\vu)
    \end{dmath*}
    \begin{dmath*}
        \geq Ka_1\big(f(\vx_{1,1}) + \innp{\nabla f(\vx_{1,1}), \vu - \vx_{1,1}} + g(\vu) \big) \\
        + \sum_{s=2}^S  \sum_{k=1}^K a_s \big(f(\vx_{s,k}) + \innp{\nabla f(\vx_{s,k}), \vu - \vx_{s,k}} + g(\vu) \big)
    \end{dmath*}
    \begin{dmath*}
    =\psi_{S,K}(\vu) - \psi_{1,0}(\vu)  + \sum_{s=2}^S  \sum_{k=1}^K a_s \innp{\nabla f(\vx_{s,k}) -\vq_{s,k}, \vu - \vx_{s,k}}
    \end{dmath*}
    \begin{dmath} \label{eq:a-coder-gap-fn-vr-1}
        \ge \psi_{S,K}(\vv_{S,K})  +  \frac{K(1 +  A_S\gamma)}{2}\|\vu - \vv_{S,K}\|^2  -  \frac{K}{2}\|\vx_0 - \vu\|^2 \\
        + \sum_{s=2}^S \sum_{k=1}^K a_s \innp{\nabla f(\vx_{s,k}) -\vq_{s,k}, \vu - \vx_{s,k}},
    \end{dmath}
\end{dgroup*}
where the first inequality is by the convexity of $f$, the second equality is by the recursive definition of $\psi_{S,K}(\vu)$, the second inequality is by the $ K(1 + A_S\gamma)$-strong convexity of $\psi_{S,K}(\vu)$ and $\vv_{S,K} = \argmin_{\vu}\psi_{S,K}(\vu)$ leading to $\psi_{S,K}(\vu) \geq \psi_{S,K}(\vv_{S,K}) +  \frac{K(1 +  A_S\gamma)}{2}\|\vu - \vv_{S,K}\|^2$.

Then using our recursive definition of the estimate sequences again, we have
\begin{align}
\;& \psi_{S,K}(\vv_{S,K}) \nonumber\\
=\;& \psi_{1,1}(\vv_{1,1})  + \sum_{s=2}^S \sum_{k=1}^K (\psi_{s,k}(\vv_{s,k}) - \psi_{s,k-1}(\vv_{s,k-1}))  \nonumber \\ 
=\;& \psi_{1,1}(\vv_{1,1})+\sum_{s=2}^S \sum_{k=1}^K\big(\psi_{s,k-1}(\vv_{s,k}) - \psi_{s,k-1}(\vv_{s,k-1})\big)\nonumber \\ 
&\;+ \sum_{s=2}^S \sum_{k=1}^K  a_s(f(\vx_{s,k}) + \innp{\vq_{s,k}, \vv_{s,k} - \vx_{s,k}} + g(\vv_{s,k}))     \nonumber\\
\ge\;&  \frac{K}{2}\|\vv_{1,1} - \vv_{1,0}\|^2  +  K a_1(f(\vx_{1,1}) + \langle \nabla f(\vx_{1,1}), \vv_{1,1} - \vx_{1,1} \rangle + g(\vv_{1,1}))  \nonumber \\ 
&\;  +  \sum_{s=2}^S  \sum_{k=1}^K  \frac{K(1 + A_{s-1}\gamma)}{2}\|\vv_{s,k} - \vv_{s,k-1}\|^2  \nonumber \\ 
&\; + \sum_{s=2}^S  \sum_{k=1}^K  a_s(f(\vx_{s,k}) + \innp{\vq_{s,k}, \vv_{s,k} - \vx_{s,k}} + g(\vv_{s,k})), \label{eq:a-coder-gap-fn-vr-2}
\end{align}
where the first equality is by $\psi_{s+1, 0 } = \psi_{s, K}$ and $\vv_{s+1,0} = \vv_{s,K}$, the second equality is by the definition of $\psi_{s,k}$, 
the last inequality is by the definition of $\psi_{1,1}(\vv_{1,1})$ and the $K(1+A_{s-1}\gamma)$-strong convexity of $\psi_{s,k-1} (s\ge 2, k\in[K]).$   
Then by Lemmas \ref{lem:vr-smooth} and \ref{lemma:Lip-quad-ub}, we have
\begin{align}
    f(\vv_{1,1}) \le\;&      f(\vx_{1,1}) + \langle \nabla f(\vx_{1,1}), \vv_{1,1} - \vx_{1,1} \rangle + \frac{L}{2} \|\vv_{1, 1} - \vx_{1,1}\|^2 \nonumber  \\ 
    \le\;&      f(\vx_{1,1}) + \langle \nabla f(\vx_{1,1}), \vv_{1,1} - \vx_{1,1} \rangle +  \frac{1}{4a_1} \|\vv_{1,1} - \vv_{1,0}\|^2, \label{eq:a-coder-gap-fn-vr-3}
\end{align}
where the last inequality is by $a_1 \le \frac{1}{4L}$ and $\vv_{1,0} = \vx_{1,1}$.

Using $\vyt_s = \frac{1}{K}\sum_{k=1}^K \vy_{s,k} = \frac{A_{s-1}}{A_s}\vyt_{s-1} +  \frac{a_s}{KA_s}\sum_{k=1}^K \vv_{s,k}$, the convexity of $g$, and $A_0 = 0$, we have  
\begin{align}
\sum_{s=2}^S \sum_{k=1}^K  a_s g(\vv_{s,k}) \ge\;& \sum_{s=2}^S K a_s g\Big(\frac{1}{K} \sum_{k=1}^K \vv_{s,k}\Big) \ge \sum_{s=2}^S\big( KA_s g(\vyt_s) - KA_{s-1}g(\vyt_{s-1})\big) \nonumber\\
=\;& KA_S g(\vyt_S) - KA_1 g(\vyt_1) \nonumber\\
=\;& KA_S g(\vyt_S) - KA_1 g(\vv_{1,1}).
\label{eq:a-coder-gap-fn-vr-4}  
\end{align}
Thus, combining Eqs.~\eqref{eq:a-coder-gap-fn-vr-1}--\eqref{eq:a-coder-gap-fn-vr-4}, we have
\begin{dmath} \label{eq:a-coder-gap-fn-vr-5}
    KA_S\bar{f}(\vu) \ge \frac{K(1 +  A_S\gamma)}{2}\|\vu - \vv_{S,K}\|^2  - \frac{K}{2}\|\vx_0 - \vu\|^2 + \frac{K}{4}\|\vv_{1,1} - \vv_{1,0}\|^2 + K a_1 f(\vv_{1,1}) \\
    + \sum_{s=2}^S \sum_{k=1}^K \Big( a_s \innp{\nabla f(\vx_{s,k}) -\vq_{s,k}, \vu - \vx_{s,k}} +  \frac{K(1 + A_{s-1}\gamma)}{2}\|\vv_{s,k} - \vv_{s,k-1}\|^2 \Big) \\
    + \sum_{s=2}^S  \sum_{k=1}^K  a_s(f(\vx_{s,k}) + \innp{\vq_{s,k}, \vv_{s,k} - \vx_{s,k}}) + KA_S g(\vyt_S). 
\end{dmath}
Then with $A_0=0$ and $\vyt_s = \frac{1}{K}\sum_{k=1}^K \vy_{s,k}$, we also have 
\begin{align}
KA_Sf(\vyt_S) =\;&  KA_1f(\vyt_1) +  K \sum_{s=2}^S (A_sf(\vyt_s) - A_{s-1}f(\vyt_{s-1})) \nonumber\\
            \le\;& K a_1 f(\vv_{1,1}) +
            \sum_{s=2}^S \sum_{k=1}^K A_sf(\vy_{s,k}) - K \sum_{s=2}^SA_{s-1}f(\vyt_{s-1}),
            \label{eq:a-coder-gap-fn-vr-6} 
\end{align}
where the last equality is by $A_1 = a_{1}, \vyt_{1} = \vv_{1,1}$. Subtracting Eq.~\eqref{eq:a-coder-gap-fn-vr-5} from \eqref{eq:a-coder-gap-fn-vr-6}  and noting that $ \bar{f}(\vyt_S) = f(\vyt_S) + g(\vyt_S)$ now completes the proof.    
\end{proof}

\begin{lemma}\label{lemma:E-bound-vr}
The error sequence $\{E_{s,k}(\vu)\}_{s\ge 2, k\in[K]}$ in Lemma~\ref{lemma:a-coder-gap-fn-vr} satisfies
\begin{dmath*}
    E_{s,k}(\vu) \le -A_{s-1}\big(f(\vyt_{s-1}) - f(\vx_{s, k}) - \big\langle \nabla  f(\vx_{s,k}), \vyt_{s-1} - \vx_{s,k}\big\rangle\big) + a_s \langle \nabla f(\vx_{s,k}) - \vq_{s,k}, \vv_{s,k} - \vu\rangle  +\Big(\frac{L{a_s}^2}{2A_s} - \frac{K(1 + A_{s-1}\gamma)}{2}\Big) \|\vv_{s,k} - \vv_{s,k-1}\|^2.
\end{dmath*}
\end{lemma}


\begin{proof}
%
%
Using Assumption~\ref{assmpt:average-smooth}, Lemma~\ref{lem:vr-smooth}, and Lemma~\ref{lemma:Lip-quad-ub}, and applying the definition of $\vy_{s, k},$ we have
\begin{dgroup*}
    \begin{dmath*}
        f(\vy_{s,k}) - f(\vx_{s,k})
    \end{dmath*}
    \begin{dmath*}
        \le \langle \nabla  f(\vx_{s,k}), \vy_{s,k} - \vx_{s,k}\rangle + \frac{L}{2} \| \vy_{s,k} -  \vx_{s,k}\|^2
    \end{dmath*}
    \begin{dmath*}
        = \big\langle \nabla  f(\vx_{s,k}),  \frac{A_{s-1}}{A_s}\vyt_{s-1} + \frac{a_s}{A_s}\vv_{k,s}   - \vx_{s,k}\big\rangle   + \frac{L{a_s}^2}{2 {A_s}^2}\|\vv_{s,k} - \vv_{s,k-1}\|^2
    \end{dmath*}
    \begin{dmath} \label{eq:E-bound-vr-1}
        = \frac{A_{s-1}}{A_s} \big\langle \nabla  f(\vx_{s,k}), \vyt_{s-1}  - \vx_{s,k}\big\rangle   +  \frac{a_{s}}{A_s} \big\langle \nabla  f(\vx_{s,k}), \vv_{s,k}  - \vx_{s,k}\big\rangle + \frac{L{a_s}^2}{2 {A_s}^2}\|\vv_{s,k} - \vv_{s,k-1}\|^2.
    \end{dmath}
\end{dgroup*}

It remains to plug Eq.~\eqref{eq:E-bound-vr-1} into the definition of $E_{s,k}(\vu)$, and rearrange.
\end{proof}

The definition of the variance reduced extrapolation point $\vq_{s,k}$ is crucial for bounding the error terms $\{E_{s,k}(\vu)\}$ from Lemma~\ref{lemma:a-coder-gap-fn-vr}.
The next three auxiliary lemmas apply the definition of $\vq_{s,k}^{(j)}$ to bound the inner product term $\langle \nabla f(\vx_{s,k}) - \vq_{s,k}, \vv_{s,k} - \vu\rangle$ in $E_{s,k}(\vu)$ when we take the expectation over all randomness in the algorithm.
We will use $\gF_{s, k, i}$ to denote the natural filtration, containing all randomness up to and including epoch $s$, outer iteration $k$, and inner iteration $i$. Note that in Algorithm~\ref{alg:vr-acc-coder-analysis}, the index of the inner iteration goes from $j = m$ to $1$, therefore inner iteration $i$ corresponds to when index of the inner iteration is $j = m - i + 1$. This detail however does not play a important role in our analysis.


\begin{lemma} \label{lemma:coder-gap-change-vr}
For all $s\ge 2$, $k\in [K]$ and $\vu\in\dom(g)$, we have
\begin{dgroup*}
    \begin{dmath*}
        a_s \E[ \langle\nabla f(\vx_{s,k}) - \vq_{s,k},  \vv_{s,k} - \vu\rangle] 
    \end{dmath*}
    \begin{dmath*}
        = \sum_{j=1}^m a_s \E[\langle \nabla^{(j)} f(\vx_{s,k}) -  \nabla^{(j)} f(\vw_{s,k, j}),  \vv_{s,k}^{(j)} - \vu^{(j)} \rangle]   
        - \sum_{j=1}^m a_{s,k-1} \E [\langle\nabla^{(j)} f(\vx_{s, k-1}) -  \nabla^{(j)} f(\vw_{s,k-1, j}),  \vv_{s,k-1}^{(j)} - \vu^{(j)} \rangle]
        - \sum_{j=1}^m a_{s,k-1} \E[\langle\nabla^{(j)} f_{t_j}(\vx_{s, k-1}) -  \nabla^{(j)} f_{t_j}(\vw_{s,k-1, j}),  \vv_{s,k}^{(j)} -  \vv_{s,k-1}^{(j)} \rangle ]  
        + \sum_{j=1}^m a_s\E[\langle\nabla^{(j)} f(\vw_{s,k, j}) - \big( \nabla^{(j)} f_{t_j}(\vw_{s,k, j}) - \nabla^{(j)} f_{t_j}(\vyt_{s-1}) + \boldsymbol{\mu}_{s}^{(j)}\big), \vv_{s,k}^{(j)} -  \vv_{s,k-1}^{(j)}\rangle],
    \end{dmath*}
\end{dgroup*}
\end{lemma}

\begin{proof}
Using the definition of $\vq_{s,k}^{(j)},$ we have 
\begin{dgroup*}
    \begin{dmath*} 
        a_s(\nabla^{(j)} f(\vx_{s,k})- \vq_{s,k}^{(j)})
    \end{dmath*}   
    \begin{dmath*}
        = a_s (\nabla^{(j)} f(\vx_{s,k}) -  \nabla^{(j)} f(\vw_{s,k,j})) + a_s( \nabla^{(j)} f(\vw_{s,k,j}) -  \vq_{s,k}^{(j)})
    \end{dmath*}
    \begin{dmath} \label{eq:coder-gap-change-vr-1}
        = a_s (\nabla^{(j)} f(\vx_{s,k}) - \nabla^{(j)} f(\vw_{s,k,j})) + a_s(\nabla^{(j)} f(\vw_{s,k,j}) -  (\nabla^{(j)} f_{t_j}(\vw_{s,k,j}) - \nabla^{(j)} f_{t_j}(\vyt_{s-1}) + \vmu_s^{(j)})) - a_{s,k-1} (\nabla^{(j)} f_{t_j}(\vx_{s, k-1}) - \nabla^{(j)} f_{t_j}(\vw_{s,k-1,j})).
    \end{dmath}
\end{dgroup*}
First, for $j \in [m]$ and any fixed $\vu^{(j)},$ we have
\begin{dgroup*}
    \begin{dmath*}
        \E[a_s\langle\nabla^{(j)} f(\vw_{s,k,j}) -  (\nabla^{(j)} f_{t_j}(\vw_{s,k,j}) - \nabla^{(j)} f_{t_j}(\vyt_{s-1}) + \vmu_s^{(j)}), \vv_{s,k}^{(j)} - \vu^{(j)}\rangle]
    \end{dmath*}
    \begin{dmath*}
        = \E[a_s\langle \nabla^{(j)} f(\vw_{s,k,j}) -  (\nabla^{(j)} f_{t_j}(\vw_{s,k,j}) - \nabla^{(j)} f_{t_j}(\vyt_{s-1}) + \vmu_s^{(j)}), \vv_{s,k}^{(j)} -  \vv_{s,k-1}^{(j)}\rangle]
        + a_s \E[ \langle \E[\nabla^{(j)} f(\vw_{s,k,j}) -  (\nabla^{(j)} f_{t_j}(\vw_{s,k,j}) - \nabla^{(j)} f_{t_j}(\vyt_{s-1}) + \vmu_s^{(j)}) | \gF_{s, k, j-1}],   \vv_{s,k-1}^{(j)} - \vu^{(j)}\rangle ]
    \end{dmath*}
    \begin{dmath} \label{eq:coder-gap-change-vr-2}
        = \E[a_s\langle \nabla^{(j)} f(\vw_{s,k,j}) -  (\nabla^{(j)} f_{t_j}(\vw_{s,k,j}) - \nabla^{(j)} f_{t_j}(\vyt_{s-1}) + \vmu_s^{(j)}), \vv_{s,k}^{(j)} -  \vv_{s,k-1}^{(j)}\rangle],
    \end{dmath}
\end{dgroup*}
where the first equality follows from $\vv_{s, k-1}^{(j)} \in \gF_{s,k,j-1}$ and the second equality follows from  $\E[ \nabla^{(j)} f_{t_j}(\vw_{s,k,j})  | \gF_{s,k,j-1} ] = \nabla^{(j)} f(\vw_{s,k,j}) $ and  $\E[\nabla^{(j)} f_{t_j}(\vyt_{s-1})| \gF_{s,k,j-1}] = \nabla^{(j)} f(\vyt_{s-1}) = \vmu_s^{(j)}$. 
Meanwhile, for $j \in [m]$ and any fixed $\vu^{(j)}$, we have
\begin{dgroup*}
    \begin{dmath*}
        \E[a_{s,k-1} \langle\nabla^{(j)} f_{t_j}(\vx_{s, k-1}) - \nabla^{(j)} f_{t_j}(\vw_{s,k-1,j}), \vv_{s,k}^{(j)} - \vu^{(j)}\rangle]
    \end{dmath*}
    \begin{dmath*}
        = \E[a_{s,k-1} \langle\nabla^{(j)} f_{t_j}(\vx_{s, k-1}) - \nabla^{(j)} f_{t_j}(\vw_{s,k-1,j}), \vv_{s,k}^{(j)} - \vv_{s,k-1}^{(j)}\rangle]
        + \E[\E[a_{s,k-1}\langle \nabla^{(j)} f_{t_j}(\vx_{s, k-1}) - \nabla^{(j)} f_{t_j}(\vw_{s,k-1,j}),  \vv_{s,k-1}^{(j)} - \vu^{(j)}\rangle| \gF_{s,k,j-1}]]
    \end{dmath*}
    \begin{dmath} \label{eq:coder-gap-change-vr-3}
        = \E[a_{s,k-1}\langle\nabla^{(j)} f_{t_j}(\vx_{s, k-1}) - \nabla^{(j)} f_{t_j}(\vw_{s,k-1,j}), \vv_{s,k}^{(j)} - \vv_{s,k-1}^{(j)}\rangle]
        + \E[ a_{s,k-1} \langle\nabla^{(j)} f(\vx_{s, k-1}) - \nabla^{(j)} f(\vw_{s,k-1,j}),  \vv_{s,k-1}^{(j)} - \vu^{(j)}\rangle], 
    \end{dmath}
\end{dgroup*}
where the last equality is by $\vv_{s,k-1}^{(j)}\in\gF_{s,k,j-1}$, $\E[\nabla^{(j)} f_{t_j}(\vx_{s, k-1})|\gF_{s,k,j-1}] = \nabla^{(j)} f(\vx_{s, k-1})$ and \\
$\E[\nabla^{(j)} f_{t_j}(\vw_{s,k-1,j}) | \gF_{s,k, j-1}] = \nabla^{(j)} f(\vw_{s,k-1,j})$. Combining Eqs.~\eqref{eq:coder-gap-change-vr-1}--\eqref{eq:coder-gap-change-vr-3} completes the proof.
\end{proof}

In the following two lemmas, we will bound the third and the fourth terms of the R.H.S. in Lemma~\ref{lemma:coder-gap-change-vr} by above using our novel Lipschitz Assumption~\ref{assmpt:Lip-const} and Assumption~\ref{assmpt:average-smooth}.


\begin{lemma} \label{lemma:vrcoder-inner-prod-bound-1}

For $s\ge 2$ and $k\in[K]$, we have
\begin{align*}
    & \; - \sum_{j=1}^m a_{s, k-1} \E\left[\inner{\nabla^{(j)} f_{t_j} (\vx_{s, k-1}) - \nabla^{(j)} f_{t_j} (\vw_{s, k-1, j})}{\vv_{s, k}^{(j)} - \vv_{s, k-1}^{(j)}}\right] \\
    & \; \le \E \left[ \frac{K (1 + A_{s-1} \gamma)}{8} \norm{\vv_{s, k} - \vv_{s, k-1}}^2 + \frac{a_{s, k-1}^4 L^2}{K A_{s, k-1}^2 (1 + A_{s-1} \gamma)} \norm{\vv_{s, k-1} - \vv_{s, k-2}}^2 \right],
\end{align*}
where $a_{s, 0} = a_{s-1}$, $A_{s, 0} = A_{s-1}$ and $a_{s, k} = a_s$, $A_{s, k} = A_s$ for $k \in [K]$.
\end{lemma}

\begin{proof}

Using Cauchy–Schwarz and Young's inequalities, we have
\begin{align} \label{eq:vrcoder-inner-prod-bound-1-1}
    & - a_{s, k-1} \E\left[\inner{\nabla^{(j)} f_{t_j} (\vx_{s, k-1}) - \nabla^{(j)} f_{t_j} (\vw_{s, k-1, j})}{\vv_{s, k}^{(j)} - \vv_{s, k-1}^{(j)}}\right] \notag\\
    & \le \E \br{ \frac{2 a_{s, k-1}^2}{K \pr{1 + A_{s-1} \gamma}} \norm{\nabla^{(j)} f_{t_j} (\vx_{s, k-1}) - \nabla^{(j)} f_{t_j} (\vw_{s, k-1, j})}^2  + \frac{K \pr{1 + A_{s-1} \gamma}}{8} \norm{\vv_{s, k}^{(j)} - \vv_{s, k-1}^{(j)} }^2} \notag\\
    & \le \E \br{ \frac{2 a_{s, k-1}^2}{K \pr{1 + A_{s-1} \gamma}} \norm{\vx_{s, k-1} - \vw_{s, k-1, j}}_{\mQ^j}^2  + \frac{K \pr{1 + A_{s-1} \gamma}}{8} \norm{\vv_{s, k}^{(j)} - \vv_{s, k-1}^{(j)} }^2} \notag\\
    & = \E \br{ \frac{2 a_{s, k-1}^2}{K \pr{1 + A_{s-1} \gamma}} \norm{\vx_{s, k-1} - \vy_{s, k-1}}_{(\mQ^j)_{\ge j+1}}^2  + \frac{K \pr{1 + A_{s-1} \gamma}}{8} \norm{\vv_{s, k}^{(j)} - \vv_{s, k-1}^{(j)} }^2},
\end{align}
where we used Assumption~\ref{assmpt:average-smooth} in the first inequality and the definitions of $\vx_{s, k-1}$ and $\vw_{s, k-1, j}$ in the last equality. Finally by including the summation and using the definition of $L$, $\vx_{s, k-1}$ and $\vy_{s, k-1}$, the first term of the above expression becomes
\begin{align} \label{eq:vrcoder-inner-prod-bound-1-2}
    \sum_{j=1}^m \E \br{ \frac{2 a_{s, k-1}^2}{K \pr{1 + A_{s-1} \gamma}} \norm{\vx_{s, k-1} - \vy_{s, k-1}}_{(\mQ^j)_{\ge j+1}}^2} 
    & = \E \br{ \frac{2 a_{s, k-1}^2}{K \pr{1 + A_{s-1} \gamma}} \norm{\vx_{s, k-1} - \vy_{s, k-1}}_{\sum_{j=1}^m (\mQ^j)_{\ge j+1}}^2} \notag\\
    & \le \E \br{ \frac{a_{s, k-1}^4 L^2}{K A_{s, k-1} ^2 \pr{1 + A_{s-1} \gamma}} \norm{\vv_{s, k-1} - \vv_{s, k-2}}^2},
\end{align}
where $a_{s, 0} = a_{s-1}$, $A_{s, 0} = A_{s-1}$ and $a_{s, k} = a_s$, $A_{s, k} = A_s$ for $k \in [K]$. Taking summation over $j$ and combining Eqs.~\eqref{eq:vrcoder-inner-prod-bound-1-1} and \eqref{eq:vrcoder-inner-prod-bound-1-2} give the lemma statement.

\end{proof}


\begin{lemma} \label{lemma:vrcoder-inner-prod-bound-2}

For $s\ge 2$ and $k\in[K]$, we have
\begin{align*}
    & \; \sum_{j=1}^m a_s \E\left[\inner{\nabla^{(j)} f(\vw_{s, k, j}) - \pr{\nabla^{(j)} f_{t_j} (\vw_{s, k, j}) - \nabla^{(j)} f_{t_j} (\vyt_{s-1}) + \nabla^{(j)} f (\vyt_{s-1})}}{\vv_{s, k}^{(j)} - \vv_{s, k-1}^{(j)}}\right] \\
    & \; \le \E \left[ \pr{\frac{2 L^2 a_s^4}{K A_s^2 (1 + A_{s-1} \gamma)} + \frac{K (1 + A_{s-1} \gamma)}{8}} \norm{\vv_{s, k} - \vv_{s, k-1}}^2 \right] \\
    & \quad \; + \frac{8 a_s^2 L}{K \pr{1 + A_{s-1} \gamma}} \E \left[ f(\vyt_{s-1}) - f(\vx_{s, k}) - \inner{\nabla f(\vx_{s, k})}{\vyt_{s-1} - \vx_{s, k}} \right]
\end{align*}

\end{lemma}

\begin{proof}
Using similar arguments in the proof of Lemma~\ref{lemma:vrcoder-inner-prod-bound-1}, we have
\begin{dgroup*}
    \begin{dmath*}
        a_s \E\left[\inner{\nabla^{(j)} f(\vw_{s, k, j}) - \pr{\nabla^{(j)} f_{t_j} (\vw_{s, k, j}) - \nabla^{(j)} f_{t_j} (\vyt_{s-1}) + \nabla^{(j)} f (\vyt_{s-1})}}{\vv_{s, k}^{(j)} - \vv_{s, k-1}^{(j)}}\right] 
    \end{dmath*}
    \begin{dmath*}
        \le \E \br{\frac{2 a_s^2}{K \pr{1 + A_{s-1}\gamma}} \norm{\nabla^{(j)} f (\vw_{s, k, j}) - \pr{\nabla^{(j)} f_{t_j} (\vw_{w, k, j}) - \nabla^{(j)} f_{t_j} (\vyt_{s-1}) + \vmu_{s}^{(j)}}}^2 + \frac{K \pr{1 + A_{s-1}\gamma}}{8} \norm{\vv_{s, k}^{(j)} - \vv_{s, k-1}^{(j)}}^2}
    \end{dmath*}
    \begin{dmath*}
        = \E \br{\frac{2 a_s^2}{K \pr{1 + A_{s-1}\gamma}} \E \br{\norm{\pr{\nabla^{(j)} f_{t_j} (\vw_{s, k, j}) + \nabla^{(j)} f_{t_j} (\vyt_{s-1})} - \pr{\nabla^{(j)} f (\vw_{s, k, j}) - \vmu_{s}^{(j)}}}^2 | \gF_{s, k, j-1}}} + \E \br{\frac{K \pr{1 + A_{s-1}\gamma}}{8} \norm{\vv_{s, k}^{(j)} - \vv_{s, k-1}^{(j)}}^2}
    \end{dmath*}
    \begin{dmath*}
        = \E \br{\frac{2 a_s^2}{K \pr{1 + A_{s-1}\gamma}} \E \br{\norm{\nabla^{(j)} f_{t_j} (\vw_{s, k, j}) - \nabla^{(j)} f_{t_j} (\vyt_{s-1})}^2 | \gF_{s, k, j-1}}} + \E \br{\frac{K \pr{1 + A_{s-1}\gamma}}{8} \norm{\vv_{s, k}^{(j)} - \vv_{s, k-1}^{(j)}}^2}
    \end{dmath*}
    \begin{dmath} \label{eq:vrcoder-inner-prod-bound-2-1}
        \le \E \br{\frac{4 a_s^2}{K \pr{1 + A_{s-1}\gamma}} \pr{\norm{\nabla^{(j)} f_{t_j} (\vw_{s, k, j}) - \nabla^{(j)} f_{t_j} (\vx_{s, k})}^2 + \norm{\nabla^{(j)} f_{t_j} (\vx_{s, k}) - \nabla^{(j)} f_{t_j} (\vyt_{s-1})}^2} + \frac{K \pr{1 + A_{s-1}\gamma}}{8} \norm{\vv_{s, k}^{(j)} - \vv_{s, k-1}^{(j)}}^2}
    \end{dmath},
\end{dgroup*}
where the first equality comes from $\E \br{\nabla^{(j)} f_{t_j} (\vw_{s, k, j}) - \nabla^{(j)} f_{t_j} (\vyt_{s-1}) | \gF_{s, k, j-1}} = \nabla^{(j)} f (\vw_{s, k, j}) - \nabla^{(j)} f (\vyt_{s-1})$ since the only randomness is in $t_j$ when conditioned at $\gF_{s, k, j-1}$, and the last inequality comes from $(a + b)^2 \le 2 (a^2 + b^2)$. In order to bound the second term in Eq.~(\ref{eq:vrcoder-inner-prod-bound-2-1}), we will include the outer summation with respect to $j$ and apply the results from Lemma~\ref{lemma:Lip-quad-ub} to get
\begin{dgroup*}
    \begin{dmath*}
        \sum_{j=1}^m \E \br{\norm{\nabla^{(j)} f_{t_j} (\vx_{s, k}) - \nabla^{(j)} f_{t_j} (\vyt_{s-1})}^2}
        = \sum_{j=1}^m \E \br{ \E \br{\norm{\nabla^{(j)} f_{t_j} (\vx_{s, k}) - \nabla^{(j)} f_{t_j} (\vyt_{s-1})}^2 | \gF_{s, k, 0}}}
    \end{dmath*}
    \begin{dmath*}
        = \E \br{ \sum_{j=1}^m \sum_{l = 1}^n \frac{1}{n} \norm{\nabla^{(j)} f_{l} (\vx_{s, k}) - \nabla^{(j)} f_{l} (\vyt_{s-1})}^2}
    \end{dmath*}
    \begin{dmath*}
        = \E \br{ \sum_{l = 1}^n  \frac{1}{n} \norm{\nabla f_{l} (\vx_{s, k}) - \nabla f_{l} (\vyt_{s-1})}^2}
    \end{dmath*}
    \begin{dmath} \label{eq:vrcoder-inner-prod-bound-2-2}
        \le \E \br{ 2L \pr{f (\vyt_{s-1}) - f (\vx_{s, k}) - \inner{\nabla f (\vx_{s, k})}{\vyt_{s-1} - \vx_{s, k}}}}
    \end{dmath},
\end{dgroup*}
where the second equality comes from $\vx_{s, k}, \vyt_{s-1} \in \gF_{s, k, 0}$ and the last inequality is by applying Lemma~\ref{lemma:Lip-quad-ub} and the definition of $f (\vx) = \frac{1}{n} \sum_{l=1}^n f_l(\vx)$. To bound the first term of Eq.~(\ref{eq:vrcoder-inner-prod-bound-2-1}), we apply similar arguments as in Lemma~\ref{lemma:vrcoder-inner-prod-bound-1} and get
\begin{equation} \label{eq:vrcoder-inner-prod-bound-2-3}
    \sum_{j=1}^m \E \br{\norm{\nabla^{(j)} f_{t_j} (\vw_{s, k, j}) - \nabla^{(j)} f_{t_j} (\vx_{s, k})}^2} \le \E \br{\frac{a_s^2 L^2}{2 A_s^2} \norm{\vv_{s, k} - \vv_{s, k-1}}^2}.
\end{equation}
Combining Eqs.~(\ref{eq:vrcoder-inner-prod-bound-2-1}) -- (\ref{eq:vrcoder-inner-prod-bound-2-3}) gives the lemma statement.

\end{proof}

\lemmavrcodererrortermboundcombined*

\begin{proof}
Combining Lemma~\ref{lemma:E-bound-vr}, \ref{lemma:coder-gap-change-vr}, \ref{lemma:vrcoder-inner-prod-bound-1}, \ref{lemma:vrcoder-inner-prod-bound-2}, setting $a_s$ such that $a_s^2 = \frac{K A_{s-1} \pr{1 + A_{s-1} \gamma}}{8 L}$ and using $\frac{A_{s-1}}{A_s} \le 1$, we have
\begin{dgroup*}
    \begin{dmath*}
        \E \br{E_{s, k} (\vu)}
        \le \sum_{j=1}^m a_s \E \br{\inner{\nabla^{(j)} f (\vx_{s, k}) - \nabla^{(j)} f (\vw_{s, k, j})}{\vv_{s, k}^{(j)} - \vu^{(j)}}} - \sum_{j=1}^m a_{s, k-1} \E \br{\inner{\nabla^{(j)} f(\vx_{s, k-1}) - \nabla^{(j)} f (\vw_{s, k-1, j})}{\vv_{s, k-1}^{(j)} - \vu^{(j)}}} - \pr{\frac{5 K \pr{1 + A_{s-1} \gamma}}{32}} \E \br{\norm{\vv_{s, k} - \vv_{s, k-1}}^2} + \pr{\frac{a_{s, k-1}^4 L^2}{K A_{s, k-1}^2 \pr{1 + A_{s-1} \gamma}}}\E \br{\norm{\vv_{s, k-1} - \vv_{s, k-2}}^2}. 
    \end{dmath*}
\end{dgroup*}

Next, by setting $a_{s, 0} = a_{s-1}$ and $a_{s, k} = a_{s}$ for $k = [K]$, we can telescope the error terms and get
\begin{dgroup*}
    \begin{dmath*}
        \sum_{s = 2}^S \sum_{k=1}^K \E \br{E_{s, k} (\vu)}
        \le \sum_{j=1}^m \sum_{s=2}^S \sum_{k=1}^K a_s \E \br{\inner{\nabla^{(j)} f (\vx_{s, k}) - \nabla^{(j)} f (\vw_{s, k, j})}{\vv_{s, k}^{(j)} - \vu^{(j)}}} - \sum_{j=1}^m \sum_{s=2}^S \sum_{k=1}^K a_{s, k-1} \E \br{\inner{\nabla^{(j)} f(\vx_{s, k-1}) - \nabla^{(j)} f (\vw_{s, k-1, j})}{\vv_{s, k-1}^{(j)} - \vu^{(j)}}} - \sum_{s=2}^S \sum_{k=1}^K \frac{5 K \pr{1 + A_{s-1} \gamma}}{32} \E \br{\norm{\vv_{s, k} - \vv_{s, k-1}}^2} + \sum_{s=2}^S \br{ \frac{K \pr{1 + A_{s-2} \gamma}}{64} \norm{\vv_{s, 0} - \vv_{s, -1}}^2 + \sum_{k=2}^K \frac{K \pr{1 + A_{s-1} \gamma}}{64} \norm{\vv_{s, k-1} - \vv_{s, k-2}}^2}
        \le 
        \sum_{j=1}^m a_S \E \br{\inner{\nabla^{(j)} f (\vx_{S, K}) - \nabla^{(j)} f (\vw_{S, K, j})}{\vv_{S, K}^{(j)} - \vu^{(j)}}} 
        - \sum_{j=1}^m a_1 \E \br{\inner{\nabla^{(j)} f(\vx_{1, 1}) - \nabla^{(j)} f (\vw_{1, 1, j})}{\vv_{1, 1}^{(j)} - \vu^{(j)}}} 
        + \frac{K \pr{1 + A_0 \gamma}}{64} \E \br{\norm{\vv_{2, 0} - \vv_{2, -1}}^2} - \frac{5 K \pr{1 + A_{S-1} \gamma}}{32} \E \br{ \norm{\vv_{S, K} - \vv_{S, K-1}}^2}.
    \end{dmath*}
\end{dgroup*}
The lemma follows by setting $A_0 = 0$, $\vv_{2, -1} = \vv_{1, 0}$, $\vx_{2, 0} = \vx_{1, 1}$ and $\vw_{2, 0, j} = \vw_{1, 1, j}$.
\end{proof}

\thmvracodermain*

\begin{proof}
Combining Lemma~\ref{lemma:a-coder-gap-fn-vr} and Lemma~\ref{lemma:vrcoder-error-term-bound-combined}, and by setting $\vy_{1, 0} = \vx_{1, 1} = \vx_0$ and $\vy_{1, 1} = \vv_{1, 1}$, we have
\begin{dgroup*}
    \begin{dmath} \label{eq:vrcoder-main-1}
        K A_S \E \br{ \bar{f} (\vyt_S) - \bar{f} (\vu)}
        \le \frac{K}{2} \norm{\vx_0 - \vu}^2 - \frac{K \pr{1 + A_S}}{2} \E \br{\norm{\vv_{S, K} - \vu}^2} - \frac{15 K}{64} \norm{\vv_{1, 1} - \vx_0}^2 - \frac{5 K \pr{1 + A_{S-1} \gamma}}{32} \E \br{\norm{\vv_{S, K} - \vv_{S, K-1}}^2} - \sum_{j=1}^m a_1 \inner{\nabla^{(j)} f(\vx_{1, 1}) - \nabla^{(j)} f(\vw_{1, 1, j})}{\vv_{1, 1}^{(j)} - \vu^{(j)}} + \sum_{j=1}^m a_S \E \br{ \inner{\nabla^{(j)} f(\vx_{S, K}) - \nabla^{(j)} f (\vw_{S, K, j})}{\vv_{S, K}^{(j)} - \vv_{S, K-1}^{(j)}}}.
    \end{dmath}
\end{dgroup*}
Using the same approach as Lemma~\ref{lemma:vrcoder-inner-prod-bound-1} and Lemma~\ref{lemma:vrcoder-inner-prod-bound-2}, we can upper bound the first inner product term by 
\begin{dgroup*}
    \begin{dmath*}
        - \sum_{j=1}^m a_1 \inner{\nabla^{(j)} f(\vx_{1, 1}) - \nabla^{(j)} f(\vw_{1, 1, j})}{\vv_{1, 1}^{(j)} - \vu^{(j)}}
        \le \frac{1}{8 K} \norm{\vv_{1, 1} - \vx_0}^2 + \frac{K}{16} \norm{\vv_{1, 1} - \vu}^2
    \end{dmath*}
    \begin{dmath} \label{eq:vrcoder-main-2}
        \le \frac{15 K}{64} \norm{\vv_{1, 1} - \vx_0}^2 + \frac{K}{8} \norm{\vx_0 - \vu}^2,
    \end{dmath}
\end{dgroup*}
where we used $(a + b)^2 \le 2(a^2 + b^2)$, $a_1 \le \frac{1}{4 L}$ and $K \ge 2$ in the last inequality. Similarly, we have
\begin{dgroup*}
    \begin{dmath} \label{eq:vrcoder-main-3}
        \sum_{j=1}^m a_S \E \br{ \inner{\nabla^{(j)} f(\vx_{S, K}) - \nabla^{(j)} f (\vw_{S, K, j})}{\vv_{S, K}^{(j)} - \vv_{S, K-1}^{(j)}}}
        \le \frac{K \pr{1 + A_{S-1} \gamma}}{8} \E \br{ \norm{\vv_{S, K} - \vv_{S, K-1}}^2} + \frac{K \pr{1 + A_{S-1} \gamma}}{64} \E \br{ \norm{\vv_{S, K} - \vu}^2},
    \end{dmath}
\end{dgroup*}
where we also used $a_S^2 \le \frac{K A_{S-1} \pr{1 + A_{S-1} \gamma}}{8L}$ here. Combining Eqs.~(\ref{eq:vrcoder-main-1})--(\ref{eq:vrcoder-main-3}) gives us our main bounds in the theorem. Lastly, recall that $\{a_s\}_{s\geq 1}$ is chosen so that ${a_s}^2 = \frac{K A_{s-1} (1 + A_{s-1}\gamma)}{8 L}$. When $\gamma = 0,$ this leads to the standard $A_s \geq \frac{k^2 K}{64 L}$ growth of accelerated algorithms by choosing $a_s = \frac{s K}{32 L}$ for $k \geq 1.$ When $\gamma > 0,$ we have $\frac{a_s}{A_{s-1}} > \sqrt{\frac{K \gamma}{8 L}},$ and it remains to use that $A_k = \frac{A_k}{A_{k-1}}\cdot \dots \cdot \frac{A_2}{A_1}\cdot A_1 = A_1 \Big(1 + \sqrt{\frac{K \gamma}{8 L}}\Big)^{k-1}$ where $a_1 = A_1 = \frac{1}{4 L}$ using the choice of $a_k$ in Algorithm~\ref{alg:acc-coder} and $A_0 = a_0 = 0$, completing the proof.
\end{proof}

\subsection{Adaptive Variance Reduced A-CODER} 

Similar to A-CODER, VR-A-CODER can adaptively estimate the Lipschitz parameter by checking the quadratic bounds between $\vy_{s, }$ and $\vx_{s, k}$ as well as between $\vyt_s$ and $\vx_{s, k}$.  For completeness, we have included the adaptive version of VR-A-CODER in Algorithm~\ref{alg:vr-acc-coder-adaptive} below.

\begin{algorithm*}[t!]
\caption{Variance Reduced A-CODER (Adaptive Version)}\label{alg:vr-acc-coder-adaptive}
\begin{algorithmic}[1]
\STATE \textbf{Input:} $\vx_0\in\mathrm{dom}(g), \gamma \geq 0, L_0 > 0, m, \{\gS^1, \dots, \gS^m\}$
\STATE \textbf{Initialization:} $\vyt_0 = \vv_{1,0} = \vy_{1,0} = \vx_{1, 1} = \vx_{0}$; $\vz_{1, 0} = \mathbf{0}$
\STATE $L_1 = L_0 / 2$
\REPEAT
    \STATE $L_1 = 2 L_1$
    \STATE $a_0= A_0 = 0$; $A_1 = a_1 = \frac{1}{4 L_0}$
    \STATE $\vz_{1, 1} = \nabla f (\vx_0)$; $\vv_{1,1} = \mathrm{prox}_{a_1 g} (\vx_0 - \vz_{1, 1})$
\UNTIL{$f(\vv_{1, 1}) \le f(\vx_0) + \inner{\nabla f(\vx_0)}{\vv_{1, 1} - \vx_0} + \frac{L_1}{2} \norm{\vv_{1, 1} - \vx_0}^2$}
\STATE $\vyt_1 = \vy_{1, 1} = \vv_{1,1}$
\STATE $\vw_{1, 1, j} = (\vx_{1, 1}^{(1)}, \ldots, \vx_{1, 1}^{(j)}, \vy_{1, 1}^{(j+1)},\ldots, \vy_{1, 1}^{(m)})$
\STATE $\vv_{2,0} = \vv_{1,1}$; $\vw_{2, 0, j} = \vw_{1, 1, j}$; $\vx_{2, 0} = \vx_{1, 1}$; $\vy_{2, 0} = \vy_{1, 1}$; $\vz_{2, 0} = \vz_{1, 1}$
\FOR{$s= 2$ to $S$}
    \STATE $L_s = L_{s-1} / 2$
    \REPEAT
        \STATE $L_s = 2 L_s$
        \STATE Set $a_s > 0$ s.t. $a_s^2 = \frac{K A_{s-1} \pr{1 + A_{s-1} \gamma}}{8 L_s}$; $A_s = A_{s-1} + a_s$  
        \STATE $a_{s, 0} = a_{s-1}$; $a_{s,1}=a_{s,2}=\cdots =a_{s,K} = a_s$
        \STATE $\vv_{s, 0} = \vv_{s-1, K}$; $\vw_{s, 0, j} = \vw_{s-1, K, j}$; $\vx_{s, 0} = \vx_{s-1, K}$; $\vy_{s, 0} = \vy_{s-1, K}$; $\vz_{s, 0} = \vz_{s-1, K}$
        \STATE $\vmu_s = \nabla f(\vyt_{s-1})$
        \FOR{$k = 1$ to $K$} 
            \STATE $\vx_{s,k} = \frac{A_{s-1}}{A_s}\vyt_{s-1} + \frac{a_s}{A_s}\vv_{s,k-1}$
            \FOR{$j = m$ to $1$} %
                \STATE $\vw_{s, k, j} = (\vx_{s,k}^{(1)}, \ldots, \vx_{s,k}^{(j)}, \vy_{s,k}^{(j+1)},\ldots, \vy_{s,k}^{(m)})$
                \STATE Choose $t$ in $[n]$ uniformly at random
                \STATE $\tilde{\nabla}_{s,k}^{(j)} = \nabla^{(j)} f_t(\vw_{s, k, j}) - \nabla^{(j)} f_t(\vyt_{s-1}) + \vmu_s^{(j)}$
                \STATE $\vq_{s,k}^{(j)} = \tilde{\nabla}_{s,k}^{(j)} + \frac{a_{s,k-1}}{a_s}(\nabla^{(j)} f_t(\vx_{s,k-1}) - \nabla^{(j)} f_t(\vw_{s, k-1, j}))$
                \STATE $\vz_{s,k}^{(j)} = \vz_{s, k-1}^{(j)} + a_s \vq^{(j)}_{s,k}$
                \STATE $\vv_{s,k}^{(j)} = \mathrm{prox}_{(A_{s-1} + \frac{a_s k}{K}) g^j}(\vx_0^{(j)} - \vz_{s,k}^{(j)} / K)$
                \STATE $\vy_{s,k}^{(j)} = \frac{A_{s-1}}{A_s} \vyt_{s-1}^{(j)} + \frac{a_s}{A_s} \vv_{s,k}^{(j)}$
            \ENDFOR
        \ENDFOR
        \STATE $\vyt_s = \frac{1}{K}\sum_{k=1}^K \vy_{s,k}$
    \UNTIL{$f(\vy_{s, k}) \le f(\vx_{s, k}) + \inner{\nabla f(\vx_{s, k})}{\vy_{s, k} - \vx_{s, k}} + \frac{L_s}{2} \norm{\vy_{s, k} - \vx_{s, k}}^2$ \\
    and $\frac{1}{n} \sum_{t = 1}^n \norm{\nabla f_t (\vx_{s, k}) - \nabla f_t (\vyt_{s-1})}^2 \le 2 L_s (f(\vyt_{s-1}) - f(\vx_{s, k}) - \inner{\nabla f(\vx_{s, k})}{\vyt_{s-1} - \vx_{s, k}})$}
\ENDFOR
\STATE \textbf{return} $\vv_{S, K}, \vyt_S$
\end{algorithmic}	
\end{algorithm*}


\end{document}

%% file: math_commands_arxiv.tex
\usepackage{breqn}
\usepackage{amsmath,amsfonts,amsthm,bm}
\usepackage{dsfont}
\usepackage{mathtools}
\usepackage{color}
\usepackage{thm-restate}
\usepackage[colorlinks=true, allcolors=blue]{hyperref}

\usepackage{algorithm,algorithmic}
\usepackage[algo2e,ruled,vlined]{algorithm2e}

\usepackage{graphicx}
\usepackage{subfigure}

\usepackage[round]{natbib}

\def\vyt{{\tilde{\bm{y}}}}

\usepackage{xcolor}
\definecolor{mypurple}{rgb}{0.67, 0.12, 0.47}

\newtheorem{lemma}{Lemma}

\newtheorem{assumption}{Assumption}

\newcommand{\innp}[1]{\left\langle #1 \right\rangle}

\newcommand{\dom}{\mathrm{dom}}

\newcommand{\dd}{\mathrm{d}}

\usepackage{psfrag}

\makeatletter
\newcommand{\subalign}[1]{%
  \vcenter{%
    \Let@ \restore@math@cr \default@tag
    \baselineskip\fontdimen10 \scriptfont\tw@
    \advance\baselineskip\fontdimen12 \scriptfont\tw@
    \lineskip\thr@@\fontdimen8 \scriptfont\thr@@
    \lineskiplimit\lineskip
    \ialign{\hfil$\m@th\scriptstyle##$&$\m@th\scriptstyle{}##$\hfil\crcr
      #1\crcr
    }%
  }%
}
\makeatother

\newcommand{\norm}[1]{\left\lVert#1\right\rVert}  
\newcommand{\inner}[2]{\left\langle#1,#2\right\rangle}  
\newcommand{\set}[1]{\left\{#1\right\}}  
\newcommand{\pr}[1]{\left(#1\right)}  
\newcommand{\br}[1]{\left[#1\right]}  

\def\eqref#1{equation~\ref{#1}}

\def\1{\bm{1}}

\def\rr{{\textnormal{r}}}

\def\vzero{{\bm{0}}}

\def\vmu{{\bm{\mu}}}

\def\vp{{\bm{p}}}
\def\vq{{\bm{q}}}

\def\vu{{\bm{u}}}
\def\vv{{\bm{v}}}
\def\vw{{\bm{w}}}
\def\vx{{\bm{x}}}
\def\vy{{\bm{y}}}
\def\vz{{\bm{z}}}

\def\mI{{\bm{I}}}

\def\mQ{{\bm{Q}}}

\DeclareMathAlphabet{\mathsfit}{\encodingdefault}{\sfdefault}{m}{sl}
\SetMathAlphabet{\mathsfit}{bold}{\encodingdefault}{\sfdefault}{bx}{n}

\def\gF{{\mathcal{F}}}

\def\gS{{\mathcal{S}}}

\def\sR{{\mathbb{R}}}

\newcommand{\E}{\mathbb{E}}

\DeclareMathOperator*{\argmin}{arg\,min}